\documentclass[amsa]{ipart}
\usepackage{algorithm,algpseudocode}
\RequirePackage{hyperref}

\startlocaldefs
\theoremstyle{plain}
\newtheorem{theorem}{Theorem}[section]

\newtheorem{definition}{Definition}[section]
\newtheorem{lemma}[definition]{Lemma}
\newtheorem{corollary}[definition]{Corollary}
\newtheorem{proposition}[definition]{Proposition}

\newtheorem{remark}[definition]{Remarks}

\def \H {{\bf H}}

\def \Z {{\bf Z}}

\def \B {{\bf B}}

\def \M {{\mathcal M}}
\def \P {{\mathcal P}}

\def \X {{\bf X}}

\def \z {{\bf z}}
\def \bw {{\bf w}}
\def \bu {{\bf u}}

\def \b0 {{\bf 0}}
\def \im {\mbox {Im}}
\def \re {\mbox {Re}}
\def \sgn {\mbox {sgn}}
\def \x {{\bf x}}
\def \y {{\bf y}}
\def \bz {{\bf z}}
\def \bw {{\bf w}}
\def \R  {{\mathbb R}}
\def \C  {{\mathbb C}}

\def\half{\frac{1}{2}}

\endlocaldefs

\pubyear{2022}
\volume{0}
\issue{0}
\firstpage{1}
\lastpage{46}
\begin{document}

\begin{frontmatter}
\title[Cauchy--Szeg\"{o} projection on Model Domains]{The Kohn--Laplacian and Cauchy--Szeg\"{o} projection \\ on Model Domains }
%\thankstext{T1}{All authors contributed equally to this work which included mathematical theory and analysis. The authors have read and agreed to the published version of the manuscript.}

\begin{aug}
    \author{\fnms{Der-Chen} \snm{Chang}\thanksref{t2}\ead[label=e1]{chang@georgetown.edu}
    \ead[label=u1,url]{http://www.foo.com}}
    \address{Department of Mathematics and Statistics, Georgetown University,\\
Washington D. C. 20057, USA. \\
and \\
Graduate Institute of Business Administration, \\
College of Management, Fu Jen Catholic University, \\
Taipei 242, Taiwan.
             \printead{e1}}            
 \author{\fnms{Ji} \snm{Li,}\thanksref{t3}\ead[label=e2]{ji.li@mq.edu.au}
    \ead[label=u1,url]{http://www.foo.com}}
    \address{School of Mathematical and Physical Sciences, Macquarie University,  \\
    NSW, 2109, Australia.
             \printead{e2}}             
   \author{\fnms{Jingzhi} \snm{Tie}\ead[label=e3]{jtie@math.uga.edu}
    \ead[label=u1,url]{http://www.foo.com}}
    \address{Department of Mathematics,
University of Georgia, \\
Athens, GA 30602, USA. 
             \printead{e3}}
\and
\author{\fnms{Qingyan} \snm{Wu}\thanksref{t4}%\thanksref{t5}
\ead[label=e4]{wuqingyan@lyu.edu.cn}
    \ead[label=u1,url]{http://www.foo.com}}
    \address{Department of Mathematics, Linyi University\\
 Shandong, 276005, China\\
             \printead{e4} \\
             \printead{u1}}
        \thankstext{t2}{Der-Chen Chang is partially supported by an NSF grant DMS-1408839 and a McDevitt Endowment Fund at Georgetown University. }
         \thankstext{t3}{Ji Li is partially supported by the Australian Research Council (ARC) through the research grant DP220100285.}
          \thankstext{t4}{Qingyan Wu is supported by the National Science Foundation of China (grant nos. 12171221 and 12071197), the Natural Science
Foundation of Shandong Province (grant nos. ZR2021MA031 and 2020KJI002).}
%\thankstext{t5}{Corresponding Author}
\end{aug}

\begin{abstract}

We study the Kohn--Laplacian and its fundamental solution on some model domains in $\mathbb C^{n+1}$,
and further discuss the explicit kernel of the Cauchy--Szeg\"{o} projections on these model domains using the real analysis method. We further 
show that these Cauchy--Szeg\"{o} kernels are Calder\'on--Zygmund kernels under the suitable quasi-metric. 
\end{abstract}


%%  Upper case for every keyword

\end{frontmatter}

\maketitle

\section{\bf Background and main results}

%\subsection{Statement of main result}
In complex analysis,  fundamental objects such as fundamental solutions for the Kohn--Laplacian,
%Laplace operators (or sublaplacian), 
Cauchy--Szeg\"{o} kernel, heat kernel, etc. are explicitly known in very few cases. But explicit solutions are very important for related analysis, especially for unbounded domains. 
In this paper, we discuss such formulae for some higher step case in higher dimension using a different approach.
\color{black}
We first review the geometry of a general real hypersurface $\M$ in $\C^{n+1}$, 
{which is unbounded and of high step,} 
and study geometrically invariant formulas for the fundamental solutions of 
the Kohn--Laplacian, Cauchy--Szeg\"{o} {kernel} and heat kernels. 
We refer to Theorems \ref{thm:fund-100}  and \ref{thm CS kernel}. 
This model domain was studied intensively by many mathematicians, especially Beals, Gaveau and Grenier
(see {\it e.g.,} Beals \cite {B}, Beals--Gaveau--Greiner \cite{BGGr1, BGGr2, BGGr3, BGGr4} 
and Calin--Chang--Greiner \cite{CCGr}).

The main result in the current paper is the estimates for those kernels 
by applying real method in harmonic analysis. 
The first step is to establish the $L^2$ estimates. For the Cauchy--Szeg\"o projection ${\bf S}$, 
this is automatic since by definition ${\bf S}:L^2(\partial \Omega_k)\to H^2({\Omega_k})$ is bounded. 
Then the next natural operator is $P(X_1,X_2)K_\lambda$. 
$K_\lambda$ is the fundamental solution for the Kohn Laplacian 
which we derived in our Sections 4 and 5, %~\eqref{eq:FS} and ~\eqref{eq:Fund-125} 
and $P(X_1,X_2)$ is a quadratic polynomial in the ``horizontal" vector fields $X_1$ and $X_2$. 
Then we can use David--Journ\'e theorem (The famous $T(1)$ Theorem) (see David--Journ\'e \cite {DJ} 
and Nagel--Rosay--Stein--Wainger \cite{NRSW}). Due to the limitation of pages, we will put detailed calculations in a forthcoming paper. 

The natural next step is to investigate whether the Cauchy--Szeg\"{o} {kernels} 
on the boundary of model domains are Calder\'on--Zygmund kernels. Note that  Diaz \cite{Di} studied this property for the Cauchy--Szeg\"{o} {kernels}  of Greiner and Stein \cite{GrS1} in domains in $\mathbb C^2$.
In this paper, by choosing a suitable control metric, we provide another proof to show that that for the model domains 
$\Omega_k$ in $\mathbb C^2$ with $k\geq1$, the 
Cauchy--Szeg\"{o} Kernels on the boundary are Calder\'on--Zygmund kernels. We refer to 
Theorem \ref{thm 7.1} for detail. Our approach can be applied to model domains 
$\Omega_k$ in $\mathbb C^{n+1}$ for general $n>1$ and $k\geq1$, with full detailed calculations in a forthcoming paper. 

Due to the notational complexity, we will not state the full details of our main theorems here and
refer the details to each section. This paper is organized as follows:
\begin{itemize}
\item In Sections 2 and 3, we review the Cauchy-Riemann geometry and subRiemannian geometry in $\mathbb C^n$, and the Kohn--Laplacian on CR-manifolds in $\C^{n+1}$, respectively;
\item In Section 4, we study the fundamental solution for Kohn--Laplacian on Siegel upper half space in $\mathbb C^{n+1}$;
\item In Section 5, we derive the fundamental solution for Kohn--Laplacian on model domains (with higher steps) in $\mathbb C^{n+1}$ and prove our first main result Theorem \ref{thm:fund-100};
\item In Section 6, we apply the result in Section 5, and obtain the explicit Cauchy--Szeg\"{o} {kernels} on the boundary of model domains, which is the second main result Theorem \ref{thm CS kernel};
\item In the last section, we prove that the 
Cauchy--Szeg\"{o} {kernels} on the boundary are Calder\'on--Zygmund kernels which is the third main result Theorem \ref{thm 7.1}.
\end{itemize}

\section{\bf Cauchy-Riemann geometry and subRiemannian geometry}

Consider
\[
\Delta_\X=\frac 1{2}\sum_{j=1}^m X_j^2+\cdots,
\]
where $\X=\{X_1,\,\dots,\,X_m\}$ are $m$ linearly independent vector fields on ${\mathcal M}_n$, an $n$-dimensional real
manifold with $m\le n$. The subspace $T_X$ spanned by $X_1,\,\dots,\,X_m$ is called the horizontal subspace, and its
complement is referred to as the missing directions.

$T_\X=T{\mathcal M}$ if and only if $\Delta_\X$ is elliptic. The operator $\Delta_\X$ is the usual Laplace--Beltrami operator.
The Newtonian potential is
\[
N(\x,\x_0)=\frac{1}{(2-n)|\Sigma_n(\x_0)|d^{n-2}(\x,\x_0)},\quad
n>2,
\]
where $|\Sigma_n(\x_0)|$ is the surface area of the induced unit
ball with center $\x_0$, and $d(\x,\x_0)$ is the Riemannian distance between $\x$ and $\x_0$. Then
\[
\Delta_XN(\x,\x_0)=\delta(\x-\x_0)+{\mathcal O}\big(d^{-n+1}(\x,\x_0)\big).
\]
\medskip

When $T_\X\ne T{\mathcal M}$, the operator is non-elliptic. Assume $\X$ satisfies {\it bracket generating condition}: ``the horizontal vector fields $\X$ and
their brackets span $T{\mathcal M}$", then
\smallskip

\noindent $(1)$. We know that from {\it Chow's Theorem} \cite{CHOW}: Given any two points $A,B\in {\mathcal M}$, there is a piecewise $C^1$ 
{\it horizontal curve} $\gamma:[0,1]\rightarrow {\mathcal M}$: 
\[ \gamma(0)=A,\quad \gamma(1)=B,
\] 
and 
\[
\dot\gamma(s)=\sum_{k=1}^m a_k(s)X_k. 
\] 
This yields a distance
and therefore a geometry which we shall call {\it subRiemannian}.
\smallskip

\noindent  $(2)$. By results of Fefferman--Phong \cite{FP1} and Fefferman--Sanchez \cite{FS1}, we know that $\Delta_\X$ is {\it subelliptic}:
\[
\big\|{\mathcal P}(X_j,X_k)u\big\|_{L^2_k}\le C\|f\|_{L^2_k},\quad
k\in \mathbb Z_+
\]
where ${\mathcal P}(X_j,X_k)$ is any quadratic polynomial in
$X_j$, $X_k$, $1\le j,k\le m$. Hence, it is hypoelliptic, {\it
i.e.,}
\[
\Delta_X u=f,\,\,\,f\in C^\infty(\M_n)\,\,\Rightarrow\,\, u\in
C^\infty(\M_n).
\]
This recovered a theorem of H\"ormander \cite{H}. 
\medskip

Set
\[
X_j=\sum_{k=1}^n a_{jk}(x)\frac{\partial}{\partial x_k},\quad
j=1,\dots, m.
\]
Then
\[
H=\frac{1}{2}\sum_{j=1}^m\Big(\sum_{k=1}^n a_{jk}(x)\xi_k\Big)^2
\]
is the Hamiltonian function on the cotangent bundle $T^*{\mathcal M}$.

A {\it bicharacteristic curve} $(\x(s),\xi(s))\in
T^*{\mathcal M}$ is a solution of the Hamilton's system:
\[
\dot x_j(s)=H_{\xi_j},\qquad \dot\xi_j(s)=-H_{x_j},
\]
with boundary conditions,
\[
x_j(0)=x_j^{(0)},\quad x_j(\tau)=x_j,\quad j=1,\dots, n,
\]
for given points $\x^{(0)},\,\x\in {\mathcal M}$.

The projection $\x(s)$ of the bicharacteristic curve on ${\mathcal
M}_n$ is a {\it geodesic}.
\medskip

\begin{remark}

This new geometry has essential differences with the Riemannian geometry. 

\noindent $(1)$ Every point $O$ of a Riemannian manifold is connected to
every other point in a sufficiently small neighborhood by a unique
geodesic. On a subRiemannian manifold there will be points
arbitrarily near $O$ which are connected to $O$ by an infinite
number of geodesics. This strange {phenomenon} was first pointed out by Gaveau (1977) and Strichartz (1986), and it brings up the question
of what ``local" means in subRiemannian geometry. Control theorists studying subRiemannian examples noticed that the Riemannian concepts of cut locus and conjugate locus behave badly in a subRiemannian context. 

\noindent  $(2)$ In Riemannian geometry the unit ball is smooth. In
subRiemannian geometry, among many distances, there is a
shortest one, often referred to as the Carnot-Carath\'eodory
distance. In subRiemannian geometry the Carnot-Carath\'eodory unit ball is singular.

\noindent  $(3)$ The exponential map is smooth in Riemannian geometry, but
often singular in subRiemannian geometry. The singularities occur
at points connected to an ``{\it origin}" by an infinite number of
geodesics. These singular points constitute a submanifold whose tangents yield the ``{\it missing directions}", that is the
directions in $T{\mathcal M}_n$ not covered by the horizontal directions.
\end{remark}
\smallskip

Suppose that we are in $3$ dimensions, $\x=(x_1,x_2,t)=(x^\prime,t)$, with
$2$ vector fields,
\[
X_1=\frac{\partial}{\partial
x_1}+2kx_2|x^\prime|^{2k-2}\frac{\partial}{\partial t},
\]
\[
X_2=\frac{\partial}{\partial
x_2}-2kx_1|x^\prime|^{2k-2}\frac{\partial}{\partial t},
\]
with $|x^\prime|^2=x_1^2+x_2^2$. The differential operator one
wants to invert is
\[
\Delta_\lambda=\frac{1}{2}\big(X_1^2+X_2^2\big)-\frac
{1}{2}i\lambda \big[X_1,X_2\big].
\]
The number given by the minimum number of brackets necessary to generate $T{\mathcal M}$ plus $1$ is referred to as the ``{\it step}" of the operator 
$\Delta_\X$. In particular, elliptic operator is {\it step $1$}, one bracket generators are  {\it step $2$}, and 
everything else is referred to as higher step.  Since
\[
[X_1,X_2]=-2k(k+1)|x^\prime|^{2(k-1)}\frac{\partial}{\partial t},
\]
$\Delta_\lambda$ is step $2$ at points $|x^\prime|\not=0$, and step $2k$ otherwise.

We first try to find the fundamental solution $K_\lambda(\x,\x_0)$ of $\Delta_\lambda$ which is the distribution solution of 
\[
\Delta_{\lambda,\x}K_\lambda(\x,\x_0)\,=\, \delta\big(\x-\x_0\big).
\]

Before we go further discussion of the operator $\Delta_\lambda$, let us review the geometry of a general real hypersurface $\M$ in $\C^n$. 
The beautiful interplay between real and complex geometry dominates the discussion.

Let us first work with $\C^n$ itself. A vector field $L$ on $\C^n$ can be expressed as a first order different
operator
\[
L=\sum_{j=1}^{2n} a_j(\x)\frac{\partial}{\partial x_j},
\]
where $a_j\in C^\infty$. In order to allow us to express the differentiations in complex notation, we start
by considering the complexified tangent bundle $\C T(\C^n)=\C\otimes T(\C^n)$. A section of this bundle is a complex
vector field:
\[
L=\sum_{j=1}^n c_j(\x)\frac{\partial}{\partial z_j}+\sum_{j=1}^n d_j(\x)\frac{\partial}{\partial \bar z_j}.
\]
We obtain two naturally defined integrable subbundles of $\C T(\C^n)$:
\[
T^{(1,0)}(\C^n)=\Big\{Z=\sum_{j=1}^n c_j(\x)\frac{\partial}{\partial z_j}\Big\}
\]
where $c_j\in C^\infty(\C^n)$. $T^{(1,0)}(\C^n)$ is integrable in the sense of
Frobenius: if $Z,W\in T^{(1,0)}(\C^n)$ $\Rightarrow$ $[Z,W]\in T^{(1,0)}(\C^n)$.
Denoted
\[
T^{(0,1)}(\C^n)=\overline {T^{(1,0)}(\C^n)}.
\]
Hence,
\[
T^{(1,0)}(\C^n)\cap T^{(0,1)}(\C^n)=\{0\}.
\]
This splitting of the tangent bundle plays a crucial rule in all aspects of complex geometry.

Let $\M$ be a smooth real hypersurface of $\C^n$, or more generally, of a complex manifold.
Again we start by tensoring with $\C$, writing $\C T\M$ for $\C\otimes T(\M)$.
We define
\[
T^{(1,0)}(\M)=T^{(1,0)}(\C^n)\cap \C T\M.
\]
As before, $T^{(0,1)}(\M)=\overline {T^{(1,0)}(\M)}$. Again we have
\[
T^{(1,0)}(\M)\cap T^{(0,1)}(\M)=\{0\}.
\]
For an abstract real manifold $\M$ we say that the subbundle $T^{(1,0)}(\M)$ defines a {\it CR structure on $\M$} if it is integrable, and its intersection with its conjugate bundle is trivial.
We call such a manifold a CR manifold. Its {\it horizontal subbundle} is the union
\[
\mathcal H(\M)=T^{(1,0)}(\M)\cup T^{(0,1)}(\M).
\]
We say that $\M$ is of {\it hypersurface type} if the fibres of $\mathcal H(\M)$ have codimension $1$ in $\C T\M$.

Given a CR manifold of hypersurface type, there is a non-vanishing differential 1-form $\eta$ such that
\[
{\mbox{ker}}(\eta)=\mathcal H(\M).
\]
We may assume that $\eta$ is purely imaginary.
\smallskip

\begin{definition}
Let $\M$ be a CR manifold of hypersurface type. The {\it Levi form} $\lambda$ is the Hermitian form defined by
\begin{equation*}
%\begin{split}
\lambda(Z,\bar W)=\big\langle \eta,[Z,\bar W]\big\rangle,\quad Z,W\in T^{(1,0)}(\M).
\end{equation*}
\end{definition}
\begin{definition}
A CR manifold of hypersurface type is {\it pseudoconvex} if all nonzero eigenvalues of $\lambda$ have the same sign.
It is called {\it strongly pseudoconvex} if $\lambda$ is definite, that is, all eigenvalues have the same non-zero sign.
\end{definition}
\smallskip

Now we want to express the Levi form on a hypersurface in terms of partial derivatives. In a neighborhood of a given point, we suppose that
\[
\M=\big\{z\in \C^n:\,\,\rho(z)=0,\,\,d\rho\ne 0\big\}.
\]
A complex vector field $Z$ is tangent to $\M$ then $Z(\rho)=0$ on $\M$. We may use
\[
\eta=\big(\partial-\bar\partial\big)(\rho).
\]
Let $Z,W\in T^{{(1,0)}}(\M)$. Then
\[
\langle \partial \rho,Z\rangle=\langle \partial \rho,W\rangle=0.
\]
By the Cartan formula for exterior derivatives,
\begin{equation*}
\begin{split}
\lambda(Z,\bar W)=&\big\langle \eta,[Z,\bar W]\big\rangle\\
=&\big\langle -d\eta,Z\wedge \bar W\big\rangle
=\big\langle \partial\bar\partial\rho,Z\wedge \bar W\big\rangle.
\end{split}
\end{equation*}
Hence we have interpreted the Levi form as the restriction of the complex Hessian of $\rho$ to sections of $T^{(1,0)}(\M)$.

Here we give a few examples.
\smallskip 

\noindent {\bf Example 1.} The zero set of
\[
\rho(z)=\sum_{j=1}^{n+1} |z_j|^2-1
\]
is the sphere $\mathbb{S}^{2n+1}$.
The horizontal subbundle $\mathcal {H}_z$ on $\mathbb{S}^{2n+1}$ decomposes into the holomorphic subspace
\[
T_z^{(1,0)}(\mathbb{S}^{2n+1})=\,{\text{span}}\{Z_1,\ldots Z_{n}\}
\]
and its conjugate
$T_z^{(0,1)}(\mathbb{S}^{2n+1})$,
where
\[
Z_j=\bar z_j\frac{\partial}{\partial z_{n+1}}-\bar z_{n+1}\frac{\partial}{\partial z_j},\quad j=1,\ldots,n.
\]
The annihilating contact 1-form is
\[
\eta=\frac{1}{2}\sum_{j=1}^{n+1} \big(\bar z_jdz_j-z_jd\bar z_j\big).
\]
The Levi form, after dividing out a nonzero factor, satisfies $\lambda(Z_j,\bar Z_k)=\delta_{jk}+\bar z_j z_k$.
Hence $\mathbb{S}^{2n+1}$ is strongly pseudoconvex.
\smallskip

\noindent {\bf Example 2.} The zero set of
\[
\rho(z)=\im(z_{n+1})-\sum_{j=1}^{n}|z_j|^2
\]
is the boundary of the Siegel upper half space, which we identify with the Heisenberg group ${\bf H}_n$. To simplify notations, we use ${\bf H}_n$ to represent $\partial\widetilde \Omega_n$ from now on. 
Since $\B_{n+1} \thickapprox \widetilde \Omega_n$ biholomorphically equivalent, the CR structures on the boundary and associated Levi forms
must be equivalent. The horizontal subbundle of ${\bf H}_n$ decomposes into holomorphic and conjugate holomorphic subspaces $T^{(1,0)}_z({\bf H}_n)={\mbox{span}}\{Z_1,\ldots, Z_{n}\}$ and $T^{{(1,0)}}_z({\bf H}_n)$ with
\[
Z_j=\frac{\partial}{\partial z_j}-2i\bar z_j\frac{\partial}{\partial z_{n+1}},\quad j=1,\ldots,n
\]
and later is spanned by the conjugate vector fields $\bar Z_1,\ldots, \bar Z_{n}$.
An annihilating contact 1-form is
\[
\eta=\frac{i}{2}\big(dz_{n+1}+d\bar z_{n+1}\big)+\sum_{j=2}^{n+1} \big(\bar z_jdz_j-z_jd\bar z_j\big).
\]
Finally, the Levi form is $\lambda(Z_j,\bar Z_k)=\delta_{jk}$.
\smallskip

\noindent {\bf Example 3.} The zero set of
\[
\rho(z)=\im(z_{n+1})
\]
is a halfspace $\Sigma$. It's horizontal space $\mathcal {H}_z(\Sigma)$ decomposes
into holomorphic and conjugate holomorphic subspaces, spanned respectively by the vector fields
\[
Z_j=\frac{\partial}{\partial z_j},\quad j=1,\ldots, n
\]
and their conjugates $\bar Z_1,\ldots, \bar Z_{n}$. In this case, the complex Hessian of the defining
function is identically equal to zero, and hence the Levi form is $\lambda(Z_j,\bar Z_k)\equiv 0$.
\medskip

\section {\bf Kohn--Laplacian} Let $\M$ be a CR-manifold in $\C^{n+1}$. 
Assume that $\{Z_1,\ldots,Z_{n}\}$
is an orthonormal basis of $T^{(0,1)}(\M)$.
Denote $\mathcal {B}^{(0,q)}(\M)$ the set of all $(0,q)$-forms on $\M$.
An element $\phi\in\mathcal {B}^{(0,q)}(\M)$ can be written as
\[
\phi=\sum_{J\in\vartheta_q}\phi_J\bar\omega^J=\sum_{J\in\vartheta_q}\phi_J\bar\omega_{j_1}\wedge
\bar\omega_{j_2}\wedge\cdots\wedge \bar\omega_{j_q}
\]
where $\vartheta_q$ is the set of all increasing $q$-tuples $J=(j_1,\ldots,j_q)$ with
$j_1<j_2<\cdots<j_q$ and $\phi_J\in\C^\infty(\M)$.
The tangential Cauchy-Riemann operator
\[
\bar\partial_b:\mathcal {B}^{(0,q)}(\M)\,\rightarrow\,\mathcal {B}^{(0,q+1)}(\M), \,\,\, q=0,\ldots,n
\]
can be written as
\[
\bar\partial_b\phi=\sum_{k=1}^{n}\sum_{J\in\vartheta_q}\bar Z_k(\phi_J)\bar\omega_k\wedge \bar\omega^J\,\in\,
\mathcal {B}^{(0,q+1)}(\M).
\]
The adjoint operator
\[
\bar\partial_b^\ast:\mathcal {B}^{(0,q)}(\M)\,\rightarrow\,\mathcal {B}^{(0,q-1)}(\M), \,\,\, q=1,\ldots,n-1
\]
can be written as
\[
\bar\partial_b^\ast \phi=-\sum_{k=1}^{n}\sum_{J\in\vartheta_q}Z_k(\phi_J)\bar\omega_k\lrcorner \bar\omega^J\,\in\,
\mathcal {B}^{(0,q-1)}(\M)
\]
where
\[
\bar\omega_k\lrcorner \bar\omega^J=0,\qquad {\mbox{if $k\not\in \vartheta_q$}}
\]
and
\[
\bar\omega_k\lrcorner \bar\omega^J=
(-1)^m\bar\omega_{j_1}\wedge\cdots\wedge \bar\omega_{j_{m-1}}\wedge \bar\omega_{j_{m+1}}\wedge\cdots\wedge \bar\omega_{j_q}
\]
if $k=j_m$. The Kohn Laplacian on $(0,q)$-forms is
\[
\square_b=\bar\partial_b\bar\partial_b^\ast+\bar\partial_b^\ast\bar\partial_b
\]
which is a self-adjoint operator defined $\mathcal {B}^{(0,q)}(\M)$.
Straightforward computation shows that the action of $\square_b$ on $\mathcal {B}^{(0,q)}(\M)$ is given by
\[
\square_b\Big(\sum_{J\in\vartheta_q}\phi_J\bar\omega^J\Big)=-\sum_{J\in\vartheta_q}\big({\mathcal L}_{n-1-2q}\phi_J\big)
\bar\omega^J
\]
where, for $\lambda\in\C$,
\[
{\mathcal L}_\lambda=\frac{1}{2}\sum_{k=1}^{n}\big(Z_k\bar Z_k+\bar Z_kZ_k\big)-i\lambda [Z_k,\bar Z_k].
\]

Now let $\Omega_k=\{(z_1,z_2)\in\C^2:\, \im(z_2)>|z_1|^{2k}\}$ be a domain in $\C^2$. The boundary $\partial\Omega_k$
is a CR manifold of hypersurface type. As before, we may change the coordinates
\[
z_1=x_1+ix_2\quad{\mbox{and}}\quad t=\re(z_2),
\]
then the operator
$\Delta_\lambda$ is exactly the Kohn Laplacian with $1-2q$, $q=0,1,2$.
\smallskip

In fact, the Kohn Laplacian $\Delta_\lambda$ has some connection with the classical mechanics. 
Consider a unit mass particle under the influence of force $F(x) = x$. Newton's law $\ddot x = x$ gives us an equation which describes the dynamics of an inverse pendulum in an unstable equilibrium, for small angle $x$.
The potential energy
\[
U(x) = - \int_0^x F(u) \, du = -\frac{x^2}{2}.
\]
The Lagrangian $L : T \R\rightarrow \R$ is the difference between
the kinetic and the potential energy
\[
L(x, \dot x ) = K - U = \frac{1}{2}\dot x^2 + \frac{1}{2} x^2 .
\]
The momentum $p = \frac{\partial L}{\partial \dot x} = \dot x$ and
the Hamiltonian associated with the above Lagrangian is obtained
using the Legendre transform: $H : T^* \R \rightarrow \R$
\[
H(x, p)= p \dot x - L(x, \dot x) = \frac{1}{2} p^2 - \frac{1}{2}
x^2.
\]
Consider the following complexification
\[
x= x_1 + i p_2, \quad p = p_1 + i x_2.
\]
Hence $H : T^* \C \rightarrow \C$ and
\begin{equation*}
\begin{split}
H(x,p) &=\frac{1}{2}p^2   - \frac{1}{2} x^2\\
&= \frac{1}{2} (p_1 + i x_2)^2 - \frac{1}{2} (x_1 + i p_2)^2\\
%&=& \frac{1}{2} (p_1 + i x_2)^2 + \frac{1}{2} i^2 (x_1 + i p_2)^2\\
%&=& \frac{1}{2} (p_1 + i x_2)^2 + \frac{1}{2} (ix_1 - p_2)^2\\
&=\frac{1}{2} (p_1 + i x_2)^2 + \frac{1}{2} (p_2 -ix_1)^2.
\end{split}
\end{equation*}
Replacing $\theta= i$,
\[
H(x, p; \theta) = \frac{1}{2}(p_1 +\theta x_2)^2 + \frac{1}{2}
(p_2 - \theta x_1)^2.
\]
Quantizing, $p_1 \rightarrow \partial_{x_1}$,  $p_2 \rightarrow
\partial_{x_2}$, $\theta \rightarrow \partial_{t}$ and hence $H
\rightarrow \Delta_0$, the Kohn Laplacian in the case $k=1$ and $\lambda=0$: 
\[
\Delta_0 = \frac{1}{2} \Big( \frac{\partial}{\partial x_1} + x_2 \frac{\partial} {\partial t}
\Big)^2 + \frac{1}{2} \Big(\frac{\partial}{\partial x_2} - x_1 \frac{\partial}{\partial_t}
\Big)^2.
\]

In general, we shall look for $K_\lambda$ in the form 
\[
K_\lambda(\x,\x_0)\,=\, \int_{\R} \frac {E(\x,\x_0,\tau)V_\lambda (\x,\x_0,\tau)}{g(\x,\x_0,\tau)}\,d\tau,
\]
where the function $g$ is a solution of the Hamilton-Jacobi equation: 
\[
\frac{\partial g}{\partial\tau}+\frac12 \big (X_1g\big )^2+\frac 12\big (X_2g\big )^2\,=\, 0,
\]
given by a modified action integral of a complex Hamiltonian problem. The associated energy 
\[
E\,=\, -\frac{\partial g}{\partial \tau}
\]
is the first invariant of motion, and the volume element $V_\lambda$ is the solution of a transport equation, which is order $1$ in the step $2$ case, $k=1$, and
order $2k$ in the higher step case, $k\ge 2$. 
\medskip

\section {\bf The step $2$ case; $k=1$}
Here
\[
X_1=\frac{\partial}{\partial x_1}+2x_2\frac{\partial}{\partial t},\qquad X_2=\frac{\partial}{\partial
x_2}-2x_1\frac{\partial}{\partial t},
\]
which are left-invariant with respect to the following Heisenberg group translation:
\[
\x\circ \y=\big(x^\prime+y^\prime, t+s+2[x_2y_1-x_1y_2]\big).
\]
where $\x=(x^\prime, t)$ and $\y=(y^\prime,s)$.  Moreover, one has $[X_1,X_2]=-4\frac{\partial}{\partial t}=-4T$.

In high dimensional case, the Siegel upper half space $\widetilde \Omega_n$ is defined as:
\begin{equation}
\label{eq:siegel-1}
\widetilde \Omega_n\,=\, \Big\{ (z^\prime, z_{n+1})\in \C^{n+1}:\,\, \im(z_{n+1})>\sum_{j=1}^{n}|z_j|^2\big\}
\end{equation}
where $z^\prime=(z_1,\ldots,z_{n})\in\C^{n}$.

The following elementary but useful identity is the key to discovering the biholomorphic maps:
\begin{equation}
\label{eq:siegel-2}
\im(\zeta)=\Big|\frac{i+\zeta}{2}\Big|^2-\Big|\frac{i-\zeta}{2}\Big|^2.
\end{equation}
With $\zeta=z_{n+1}$ we plug~\eqref{eq:siegel-2} into~\eqref{eq:siegel-1} and rewrite to obtain
\begin{equation}
\label{eq:siegel-3}
\Big|\frac{i+z_{n+1}}{2}\Big|^2>\sum_{j=1}^{n}|z_j|^2+\Big|\frac{i-z_{n+1}}{2}\Big|^2.
\end{equation}
After dividing by $\Big|\frac{i+z_{n+1}}{2}\Big|^2$ and changing notation, inequality~\eqref{eq:siegel-3} becomes
\[
1>\sum_{j=1}^{n+1}|w_j|^2,
\]
the defining property of the unit ball. The explicit mapping is given by
\[
w_j\,=\, \frac{2z_j}{i+z_{n+1}},\quad 1\le j\le n
\]
and
\[
w_{n+1}\,=\, \frac{i-z_{n+1}}{i+z_{n+1}}.
\]
It is easy to check that this transformation $z\mapsto w$ is biholomorphic from $\widetilde \Omega_n$ to ${\bf B}_{n+1}$.

We now describe the analogous situation on $\partial{\widetilde \Omega}_n$ of the Siegel upper half space:
\[
\partial{\widetilde\Omega}_n\, =\, \Big\{ (z^\prime, z_{n+1})\in \C^{n+1}:\,\, \im(z_{n+1})\,=\, \sum_{j=1}^{n}|z_j|^2\Big\}.
\]
The variable $z_{n+1}$ plays a different role, and hence for $\z\in \C^{n+1}$ we write $\z=(z^\prime, z_{n+1})$.
Two natural families of biholomorphic self-maps of $\widetilde\Omega_n$ are the {\it dilations} $\delta_\rho:\widetilde\Omega_n\rightarrow\widetilde \Omega_n$, for $\rho>0$, given by
\[
\delta_\rho(z^\prime, z_{n+1})\,=\, \big(\rho z^\prime, \rho^2 z_{n+1}\big),
\]
and the {\it rotations} $R_A:\widetilde \Omega_n\rightarrow\widetilde \Omega_n$, for $A\in U(n)$, given by
\[
R_A(z^\prime, z_{n+1})\,=\, \big(A(z^\prime), z_{n+1}\big).
\]
To introduce an analogue of translation we consider
\[
\tau_\x:\widetilde \Omega_n\,\rightarrow\, \widetilde \Omega_n,
\]
for $\x=(z^\prime, t)\in \C^n\times \R$, given by
\[
\tau_\x(w^\prime, w_{n+1})\,=\,\Big( w^\prime+z^\prime ,w_{n+1}+t+2i\langle z^\prime ,w^\prime \rangle+i|z^\prime |^2, \Big).
\]

All of the preceding maps extend to self-maps of the boundary $\partial{\widetilde \Omega}_n$.
The action of the family $\{\tau_\x:\, \x=(z, t)\, \in \, \C^{n}\times \R\}$
on $\widetilde \Omega_n\times \partial{\widetilde\Omega}_n$
is faithful and the action on $\partial{\widetilde\Omega}_n$ is simply transitive. We obtain a useful identification of $\partial{\widetilde\Omega}_n$ with $\C^n\times \R$. By this method we equip $ \C^{n}\times\R $ with a group
law:
\[
(\x,\y)\,\mapsto\, \x\circ \y,
\]
characterized by the identity $\tau_\x\cdot \tau_\y=\tau_{\x\circ \y}$. Here
\[
\x\circ \y\,=\, \Big(x^\prime+y^\prime, t+s+2\sum_{j=1}^n\big [x_{n+j}y_j-x_jy_{n+j}\big]\Big).
\]
The resulting space is the {\it Heisenberg group} ${\bf H}_n$ as we mentioned before.

Finally, we observe that the group of biholomorphic automorphisms of $\widetilde \Omega_n$ which fix the point at $\infty$ is generated by dilations, rotations, and translations.
\medskip

\subsection {\bf Lagrangian formalism}
In order to simplify notations, let us return to the case $\C\times \R\thickapprox \R^3$. 
We shall associate a Lagrangian $L:T\R^3\rightarrow\R$ with the Hamiltonian:
\begin{equation*}
H(\x,\xi)\,=\, \frac{1}{2}(\xi_1+2x_2\theta)^2+\frac{1}{2}(\xi_2-2x_1\theta)^2.
\end{equation*}
This can be done by using the {\it Legendre transform} in $(\dot x_1, \dot x_2, \dot t)$. It is known that
\[
H(\x,\dot\x)\,=\, \frac{1}{2}(\dot x_1^2+\dot x_2^2)+\theta(\dot t-2x_2\dot x_1+2x_1\dot x_2).
\]
Using polar coordinates, the Lagrangian:
\[
L\,=\, \frac{1}{2}(\dot r^2+r^2\dot\phi^2)+\theta(\dot t+2r^2\dot\phi).
\]
A computation shows
\begin{equation*}
\begin{split}
\frac{d}{ds}\frac{\partial L}{\partial \dot r}= \ddot r,\quad
&\frac{\partial L}{\partial r}=r\dot\phi \big(\dot\phi+4\theta\big),\\
\frac{\partial L}{\partial \dot \phi}= r^2\dot\phi+2\theta r^2,\quad
&\frac{\partial L}{\partial \phi}=0,\\
\frac{\partial L}{\partial \dot t}=\theta,\quad
&\frac{\partial L}{\partial t}=0,
\end{split}
\end{equation*}
and hence $r(s)$, $\phi(s)$ and $\theta$ satisfy the {\it Euler-Lagrange system}
\begin{equation}
\label{eq:EL1}
\begin{cases} \ddot r=r\dot \phi\big(\dot\phi+4\theta\big)\\
r^2\big(\dot\phi+2\theta\big)=C({\mbox{constant}})\\
\theta=\theta_0={\mbox{constant}}
\end{cases}
\end{equation}
If the geodesic starts at the origin, $r(0)=0$ $\Rightarrow$ $C=0$. Then 2nd equation of~\eqref{eq:EL1} yields $\dot\phi=-2\theta$.
The Euler-Lagrange system becomes
\begin{equation}
\label{eq:EL2}
\begin{cases} \ddot r=-4\theta^2r\\
\dot\phi=-2\theta\\
\theta=\theta_0({\mbox{constant}}).
\end{cases}
\end{equation}
When $\theta_0=0$, the system~\eqref{eq:EL2} becomes
\begin{equation*}
\begin{cases} \ddot r=0\\
\dot\phi=0\\
\theta_0=0.
\end{cases}
\end{equation*}
\medskip

\begin{proposition}
Given a point $P(0,x^\prime)$, there is a {\bf unique} geodesic between the origin and $P$.
It is a straight line in the plane $\{t=0\}$ of length $|x^\prime|=\sqrt{x_1^2+x_2^2}$, and it is
obtained for $\theta=0$.
\end{proposition}

Now let us move the end point $P$ away from the $x^\prime$ with $|x^\prime|\ne 0$. Consider the boundary conditions:
\begin{equation*}
\begin{split}
x^\prime(0)=0,\quad &t(0)=0,\quad \phi(0)=\phi_0\\
|x^\prime(\tau)|=R,\quad & t(\tau)=t,\,\,\,\, \phi(\tau)=\phi_1.
\end{split}
\end{equation*}
We may choose $\phi_0=0$. One has
\begin{equation*}
\begin{cases}  \dot t=-2r^2\dot\phi\\
\dot\phi=-2\theta
\end{cases} \quad \Rightarrow\quad
\dot t=4\theta r^2>0.
\end{equation*}
This implies that $t(s)$ is increasing and if $t(0)=0$, then $t(\tau)>0$.
\medskip

\begin{lemma}
The following relations take place among the boundary conditions:
\begin{equation*}
\begin{split}
\phi_1=&-2\theta\tau,\\
\sin^2\phi_1=& 4\theta^2R^2,\\
t=&\frac{1}{4\theta^2}\frac{\sin(2\phi_1)-2\phi_1}{2},\\
\frac{|t|}{R^2}=&-\mu(\phi_1)=\mu(2\theta \tau),
\end{split}
\end{equation*}
where
\[
\mu(z)=\frac{z}{\sin^2 z}-\cot z.
\]
\end{lemma}
\medskip

\begin{lemma}\label{l:mu}
$\mu$ is a monotone increasing diffeomorphism of the interval
$(-\pi, \pi)$ onto $\R$. On each interval $(m \pi, (m+1)
\pi )$, $m\in\mathbb N$, $\mu$ has a unique critical point $x_m$.
On this interval $\mu$ decreases strictly from $+\infty$ to
$\mu(x_m)$ and then increases strictly from $\mu(x_m)$ to
$+\infty$. Moreover
\[
\mu(x_m) + \pi < \mu(x_{m+1}), \quad m\in\mathbb N.
\]
\end{lemma}
\medskip 

%\centerline {\bf Geometry}
\begin{theorem}
\label{th:Geo} On ${\bf H}_1$, there are a {\bf finite} number of
geodesics connecting $\x$ to ${\bf
0}\,\Leftrightarrow\,x^\prime\ne 0$. They are parametrized by
solutions $\tau_m$, $m=1,\ldots,N$, of the transcendental
equation
\begin{equation} \label{eq:mu}
|t|\,=\, \mu(\tau_m)|x^\prime|^2=\mu(\tau_m)(x_1^2+x_2^2),
\end{equation}
where
\[
\mu(z)=\frac{z}{\sin^2z}-\cot z.
\]
The length
\[
d_m(\x)= \sqrt{\nu(\tau_m)\big(|x^\prime|^2+|t|\big)}
\]
where
\[
\nu(z)=\frac{z^2}{z+\sin^2z-\sin\,z\cos\,z}.
\]
Here $\tau_1,\ldots,\tau_N$ {are} the solutions of equation~\eqref{eq:mu}. The shortest $d_1(\x)$ is called the Carnot-Carath\'eodory distance.
\end{theorem}
\smallskip 

It is easy to see that the number of geodesics increasing without bound as $\frac{|t|}{|x^\prime|^2}\rightarrow\infty$. We have the following theorem. 
\smallskip

\begin{theorem}
Every point of the line $(0, t)$ is connected to the origin by an
{\bf infinite} number of geodesics with lengths
\[
d_m^2=m\pi|t|,\qquad m\in {\mathbb N}.
\]
For each length $d_m$, the equations of geodesics are
\begin{equation*}
\begin{split}
r^{(m)}(\phi)=&\sqrt{\frac{|t|}{m\pi}}\sin(\phi-\phi_0)\\
t^{(m)}(\phi)=&\frac{|t|}{2m\pi}\Big(\sin(2(\phi-\phi_0))-2(\phi-\phi_0)\Big).
\end{split}
\end{equation*}
The geodesic of that length are parametrized by the circle ${\mathbb S}^1$.
\end{theorem}
\medskip 
The projection of the $m$-th geodesic on the $(x_1,x_2)$-plane is a circle with radius
\[
R_m=\frac{1}{2}\sqrt{\frac{|t|}{m\pi}}
\]
and area
\[
\sigma_m=\frac{|t|}{4m}.
\]
\medskip

\subsection{\bf Fundamental solution for the operator $\Delta_\lambda$ for $k=1$} 

Consider the complex action function
\begin{equation*}
\begin{split}
g(x^\prime,t,\tau)=&-it+\int_0^\tau\Big\{\sum_{j=1}^2\dot x_j\xi_j-H\Big\}ds\\
=&-it+\coth(2\tau)|x^\prime|^2\,=\, -it+\coth(2\tau)\big( x_1^2+x_2^2\big). 
\end{split}
\end{equation*}
Like the classical action, the complex action $g$ satisfies the Hamilton-Jacobi equation:
\[
\frac{\partial g}{\partial\tau}+H\big(\x,\nabla_\x g\big)=0.
\]
Set
\[
f(\x,\tau)=\tau g(\x,\tau)=\tau\big(-it+\coth(2\tau)|x^\prime|^2\big).
\]
This is so called a ``{\it modified complex distance}".
\medskip

\begin{theorem}
\label{th:crt} Let $\tau_1(\x)$, $\tau_2(\x)$,... denote the
critical points of $f(\x,\tau)$, {\it i.e.,}
\[
\frac{\partial f}{\partial\tau}(\x,\tau_j(\x))=0.
\]
Then
\[
f(\x,\tau_j(\x))=\frac{1}{2}d_j^2(\x).
\]
\end{theorem}
\medskip

It is expected that the fundamental solution has the following form:
\begin{equation}
\label{eq:fund100}
K_\lambda(\x,\y)=\int_\R \frac{E(\x,\y,\tau)V_\lambda(\x,\y,\tau)}{g(\x,\y,\tau)}d\tau
\end{equation}
Since $\Delta_\lambda$ is left-invariant, we may set $\x_0={\bf 0}$ in $K_\lambda(\x,\x_0)$:
\[
K_\lambda(\x,{\bf 0})=\int_{\R}\frac {E(\x,\tau)V_\lambda(\x,\tau)}{g(\x,\tau)}d\tau,
\]
where $E$ and $V_\lambda$ can be calculated explicitly:
\[
E(\x,\tau)=-\frac{\partial g}{\partial\tau}=\frac{2(x_1^2+x_2^2)}{\sinh^2(2\tau)},
\]
and
\[
V_\lambda(\x,\tau)=-\frac1{4\pi^2}e^{-2\lambda \tau}\frac{\sinh(2\tau)}{x_1^2+x_2^2}.
\]
Using contour integration, one obtains
\begin{equation}
\label{eq:FS}
\begin{split}
K_\lambda(\x)\,=\, -\frac{\Gamma\big(\frac{1+\lambda}{2}\big)\Gamma\big(\frac{1-\lambda}{2}\big)}
{4\pi^2}\times
(|x^\prime|^2-it)^{-\frac{1+\lambda}{2}}(|x^\prime|^2+it)^{-\frac{1-\lambda}{2}}.
\end{split}
\end{equation}
This is the famous Folland--Stein formula \cite{FS}. Moreover, if $\lambda\ne \pm(n+2\ell)$, $\ell\in\mathbb Z_+$,
\[
\big\|\mathcal {P}(X_j,X_K)K_\lambda(f)\big\|_{L^p_k}\le C\|f\|_{L^p_k}
\]
and hence $\|K_\lambda(f)\|_{L^p_{k+1}}\le C\|f\|_{L^p_k}$ 
for $k\in\mathbb Z_+$ and $1<p<\infty$. Here $\mathcal {P}(X_j,X_K)$ is any quadratic polynomial in $X_j,X_k$. For detailed discussion, we refer readers to the books by Calin--Chang--Greiner \cite{CCGr};  Calin--Chang--Funrutani--Iwasaki \cite{CCFI} and  Chang--Tie \cite{CT}. 
\medskip 

\section{\bf The higher step case; $k\ge 2$}
As we know, there is no group structure on the boundary $\partial\Omega_k$ in this case. Moreover, solutions for Hamilton's system can not be written into elementary functions. Let us look at the simplest higher step case in $\C^2$, {\it i.e.,} $k=2$. In this case, one has 
\[
X_1= \frac {\partial}{\partial x_1} + 4x_2|x^\prime|^2 \frac
{\partial}{\partial t},\quad X_2= \frac {\partial}{\partial
x_2} - 4x_1|x^\prime|^2 \frac {\partial}{\partial t},
\]
with $|x^\prime|^2 = x_1^2 + x_2^2$. Since 
%$[X_1,X_2]\,=\, -16|x^\prime|^2  \frac {\partial}{\partial x_0}$, $[X_,[X_1,X_2]]\,=\, -32 x_1 \frac {\partial}{\partial x_0}$, and 
$$
[X_1,X_2]\,=\, -16|x^\prime|^2  \frac {\partial}{\partial t}, ~~  [X_1,[X_1,X_2]]\,=\, -32 x_1 \frac {\partial}{\partial t}, %~ {\text{and}} ~
~~[X_1,,[X_1,X_2]]]\,=\, -32 \frac {\partial}{\partial t}, 
$$  
hence $\partial\Omega_2$ is step $4$ away from the ${(x_1,x_2)}$-plane. Denote $x_0=t$. After length calculations, we have the following result. 
\medskip

\begin{theorem}\label{t:1}
Let $P(x_1,x_2,t)$ be a point of $\R^3$. Then
\smallskip

\noindent $(i)$. If $t=0$, then there is a unique geodesic between
the origin and $P$. It is a straight line in the $(x_1,x_2)$-plane of length
$\sqrt{x_1^2 + x_2^2}$.
\smallskip

\noindent $(ii)$. If $|x^\prime| = 0$ and $t \not=0$ there are infinitely many geodesics between the origin and $P$.
\smallskip

The subRiemannian geodesics that join the origin to a point
$(0,0,t)$ have lengths $d_1$, $d_2$, $d_3,\ldots$, where
\[
(d_m)^4 = \frac { m^3 K^4}{4 Q } |t|
\]
with
\[
K = \int_0^1 \frac{d\omega}{\sqrt{ (1 - \omega^2) ( 1 - k^2 \omega^2)}} 
\] 
being the complete Jacobi integral and
\[
k=\frac{\sqrt{2}}{4}\big(\sqrt 3-1\big), \quad Q= \frac{1}{4} \frac{ \Gamma(1/6)}{\Gamma(2/3) } \sqrt{\pi}.
\]
For each length $d_m$, the geodesics of that length are
parameterized by the circle ${\mathbb S}^1$.
\smallskip

\noindent $(iii)$. If $P$ is away from the $t$-axis and $x'$-plane with
$0< \frac{|t|}{(x_1^2+x_2^2)^2} < \infty $, then there are not less than $2m-1$
and not more than $2m+1$ subRiemannian geodesics between the origin
and the point $P$, where the integer $m$ is defined by
\[
\big(m - \frac{1}{2}\big) Q \,<\, \frac{3}{4} \frac{|t|}{(x_1^2+x_2^2)^2} \,\leq\, 
\big(m + \frac{1}{2}\big) Q.
\]
\end{theorem}
\smallskip

We can also compute the lengths of geodesics connecting the origin and the point $(x_1,x_2, t)$. Here is just state the result. 
\smallskip

\begin{theorem}\label{thm:step4-2} Let
$\tau_{j}$ be the critical points of the modified complex action
$f(\tau) = \tau g(\tau)$. Setting $\zeta_j = \digamma(i
\tau_{j})$, the lengths of the geodesics between the origin and
the point $(x_1,x_2, t)$, $|x^\prime |\neq 0$ are given by
$$ \ell^4 _j = \nu(\zeta_j) \big( |t| + |x^\prime |^4 \big).$$
Here
\begin{equation*}
\begin{split}
\digamma(z)  =&\,  \frac{1+ \sqrt{3}}{4^{1/3}}z-3^{1/2} \,
2^{-2/3}z-\frac{1}{2} \tan^{-1} \Bigg( \frac{ \hbox{sd} ( 2^{4/3} \, 3^{1/4} z) }{2 \cdot 3^{1/4}} \Bigg) \\
&\, +\frac{1}{2} \bigg\{ u E( am^{-1} \omega, k')
+\frac {i}{2}\log  \frac {\theta_4(x-i y)}{\theta_4(x+iy)}\, + u\Big[ \big( \frac{2\pi}{\Gamma(1/6) \Gamma(1/3)} \big)^2 -
\frac{3 - \sqrt{3}}{6}\Big] am^{-1} \omega
 \bigg\}
\end{split}
\end{equation*}
with
\[
u = 2^{4/3}\, 3^{1/4} \, z,\quad{\mbox{and}}\quad am^{-1} \omega
=\hbox{sn}^{-1} (\sqrt{3}-1, k').
\] Here
\[
am(v)=\int_0^v {\hbox{dn}}(\gamma) \,d\gamma
\]
and $\theta_4$ stands for Jacobi's zeta function.
\end{theorem}
\medskip

\subsection{\bf Fundamental solution for the sub-Laplacian $\Delta_\lambda$ with $\lambda=0$}

Since the defining function for $\partial\Omega_k$ is $\big \{\im(z_2)=|\phi(z_1)|^2=|z_1|^{2k} \big\}\,\subset\, \C^2$ with $k=2,3,\ldots$, thus it is convenient to use polar coordinates 
$(r(\varsigma),\omega(\varsigma))$, $\varsigma \in [0,\tau]$, for the variable $z=x_1+ix_2\, \in \C$. As usual, set 
\[
r^2\,=\, x_1^2+x_2^2,\qquad \omega \,=\, \frac{1}{2i}\log\,\frac{x_1+ix_2}{x_1-ix_2}. 
\]
It follows that $\big(r^2\big)^\cdot \,=\, 2r\dot r\,=\, 2x_1\dot x_1+2x_2\dot x_2$, 
and $r^2\dot \omega \,=\, x_1\dot x_2-x_2\dot x_1$. 
Thus, $\big(r^2\dot\omega \big)^\cdot \,=\, x_1\ddot x_2-x_2\ddot x_1$. Since $\dot x_1=\frac{\partial H}{\partial \xi_1}$ and  
$\dot x_2=\frac{\partial H}{\partial \xi_2}$ , we know that 
\[
\ddot x_1\,=\, -4i\big(r^2\rho^{\prime\prime}(r^2)+\rho^\prime (r^2)\big)\dot x_2,\qquad 
\ddot x_2\,=\, 4i\big(r^2\rho^{\prime\prime}(r^2)+\rho^\prime (r^2)\big)\dot x_1,
\]
where $\rho(v)=v^k$ and $\rho^\prime(v)=kv^{k-1}$. 
This yields 
\begin{eqnarray*}
\big(r^2\dot\omega \big)^\cdot &=& 2i\big(r^2\rho^{\prime\prime}(r^2)+\rho^\prime (r^2)\big)\times \big(2x_1\dot x_1+2x_2\dot x_2\big)\\
&=& 2i\big(r^2\rho^{\prime\prime}(r^2)+\rho^\prime (r^2)\big)\times t\big(r^2\big)^\cdot \, =\, 2i\big(r^2 \rho^\prime (r^2)\big)^\cdot. 
\end{eqnarray*}
Hence, there is a function $\Omega=\Omega(x,x_0,\tau)$, the {\it angular momentum}, constant on the bicharacteristic, such that 
\[
r^2(\varsigma) \dot\omega (\varsigma)\,=\, i\Big(2r^2(\varsigma)  \rho^\prime\big (r^2(\varsigma)\big)-\Omega \Big)\,=\, iW\big (r^2(\varsigma)\big)\Big|_{\varsigma=r^2}. 
\]
It follows that  
\begin{equation}
\label{eq:angular-2}
W (v)\,=\, W(v,\Omega) \,=\, 2v \cdot \rho^\prime(v)-\Omega\,=\, 2kv^k-\Omega. 
 \end{equation}
 Thus 
\begin{eqnarray*}
\big(r \dot r\big)^2+\big( r^2\dot\omega \big)^2  &=& \big(x_1\dot x_1+x_2\dot x_2\big)\big( x_1\dot x_2-x_2\dot x_1\big)^2\\
&=& \big(x_1^2+x_2^2\big)\big( \dot x_1^2+x_2^2\big)\,=\, 2r^2 E.
\end{eqnarray*}
Hence 
 \begin{equation*}
% \label{eq:angular-10}
 \Big(\frac 12\big(r^2\big)^\cdot\Big)^2\,=\, 2r^2E-\big( r^2(\varsigma) \dot\phi\big)^2 \,=\, 2r^2(\varsigma) E+W^2(r^2(\varsigma)).
 \end{equation*} 
Letting $v=r^2(\varsigma)$, we obtain
 \begin{equation}
 \label{eq:angular-10}
\frac {dv}{d\varsigma}\,=\, 2\sqrt{2E v+W^2(v)}\qquad \Rightarrow \qquad 
2\, d\varsigma \,=\, \frac{dv}{\sqrt{2E v+W^2(v)}}.
\end{equation} 
Denote $r=r(\tau)$, $r_0=r(0)$, $\omega=\omega(\tau)$ and $\omega_0=\omega(0)$. The function $E$ and $\Omega$ which are constant on each bicharacteristic, and determined implicitly by integrating (\ref{eq:angular-10})
\[
\varsigma\,=\, \int_0^\varsigma dv\,=\, \frac 12\int_{r_0^2}^{r^2(\varsigma)}\frac{dv}{\sqrt{2E v+W^2(v)}},
\]
and
\[
\omega(\varsigma)-\omega_0\,=\, \int_0^\varsigma \dot\omega(v)dv \,=\,
\frac 12\int_{r_0^2}^{r^2(\varsigma)}\frac{W(v) }{\sqrt{2E v+W^2(v)}}\frac {dv}{v}.
\]

On the other hand, from previous calculations, we know that 
\begin{eqnarray*}
\dot x_1\xi_1+\dot x_2\xi_2-H &=& \zeta_1\xi_1+\zeta_2\xi_2-\frac 12\big(\zeta_1^2+\zeta_2^2\big)\\
&=& \frac 12\big(\zeta_1^2+\zeta_2^2\big)+\zeta_1\big(\xi_1-\zeta_1\big)+\zeta_2\big(\xi_2-\zeta_2\big)\\
&=& E+2i\rho^\prime (r^2)\cdot r^2\dot\phi\,=\, E+2\rho^\prime (r^2)\cdot W(r^2).
\end{eqnarray*}
If follows that 
\begin{equation}
\label{eq:action}
\begin{split}
&\int_0^\tau \Big[\sum_{j=1}^2 \dot x_j(\varsigma)\xi_j(\varsigma) -H\big (x(\varsigma),t(\varsigma);\xi(\varsigma),\theta(\varsigma) \big )\Big] \,d\varsigma \\
\,=&\, E\tau +2\int_0^\tau \rho^\prime (r^2(\varsigma))W(r^2(\varsigma))d\varsigma \,=\, E\tau +\int_{|w|^2}^{|z|^2} 
\frac {\rho^\prime (v) W(v)}{\sqrt{2E v+W^2(v)}}dv.
\end{split} 
\end{equation}

It is impossible to calculate $W$ and $\Omega$ explicitly, but we know their analytic properties, and $g$ and $V_\lambda$  may be found in terms of $E$ and $\Omega$.
When $\phi(z)=z^k$, one has $\rho(v)=v^k$ and $\rho^\prime(v)=kv^{k-1}$. Moreover, from (\ref{eq:angular-2}), we know that $W(v)=2kv^k-\Omega$.
Hence, $\rho^\prime(v)\,=\, \frac{1}{2k}W^\prime (v)$. 
In this case, 
\begin{eqnarray*}
\int_{|w|^2}^{|z|^2}  \frac {\rho^\prime (v) W(v)}{\sqrt{2E v+W^2(v)}}\, dv &=&  
\frac{1}{2k}\int_{|w|^2}^{|z|^2} 
\frac {W^\prime (v) W(v)}{\sqrt{2E v+W^2(v)}} \, dv\\
&=& \frac{1}{2k}\cdot \sgn(\tau)\cdot \sqrt{2E v+W^2(v)}\Big|_{v=|w|^2}^{v=|z|^2}-\frac {E\tau}{k}.
\end{eqnarray*}
Combining the above result with (\ref{eq:action}), we have 
\begin{eqnarray*}
g&=&\, -i(t-s)+\Big(1-\frac 1k\Big)E\tau\cr 
&&\,\,\, +\frac {1}{2k}{\mbox {sgn}}(\tau)\Big\{\big(2E|z|^2+W(|z|^2)^2\big)^{\frac 12}-
\big(2E|w|^2+W(|w|^2)^2\big)^{\frac 12}\Big\},
\end{eqnarray*}
where one uses the principal branch of the square roots. We define 
\begin{align}\label{P}
\P\,=\, \frac{2^{\frac 1k}z\bar w}{\big[|z|^{2k}+|w|^{2k}-i(t-t_0)\big]^{\frac 1k}}.
\end{align}
This expression in square brackets has non-negative real part and we take the principal branch of the root. Thus $\P^k=u_+$ band $\bar \P^k=u_-$.
Note that $|\P|\le 1$ with equality only when $|z|=|w|$, $t=t_0$. We also define functions 
\begin{equation}
\label{eq:Fund-000}
\begin{split} 
F^{\pm}_{\lambda,\ell}(\P_+,\P_-)&= \int_0^1\Big\{ \Big(\varsigma^{\frac 1k} \P_\pm\Big)^\ell \varsigma^{-\frac{1+\lambda}{2}}(1-\varsigma)^{-\frac{1-\lambda}{2}} \\
&\qquad \times \big(1-(\P_+\P_-)^k\varsigma\big)^{-\frac{1+\lambda}{2}} \big(1-\P_\pm^k \varsigma\big)^{-1}\Big\} d\varsigma,
\end{split}
\end{equation}
for $\ell=0,1,2,\ldots, k$. These functions are holomorphic functions of their arguments so long as $\P_\pm^k$ and $(\P_+\P_-)^k$ do not belong to the interval $[1,\infty)$. 
\smallskip

\noindent {\bf Remark.} 
The expression (\ref{eq:Fund-000}) 
The functions $F^{\pm}_{\lambda,\ell}(\P_+,\P_-)$ of (\ref{eq:Fund-000}) can be identified as generalized hypergeometric functions of Appell \cite{App} which are real analytic functions of $z$, $\bar z$, $w$, $\bar w$, $t$ and $t_0$ is the region 
 $|\P|<1$, {\it i.e.,} everywhere except $t=t_0$ and $|z|=|w|$. Furthermore, they do not extend smoothly to the boundary but in pairs they do. 
 More precisely, the functions $F^{+}_{\lambda,\ell}+F^{-}_{-\lambda, k-\ell}$, $\ell=0,1,\ldots, k$, extend to an analytic function of  $z$, $\bar z$, $w$, $\bar w$, $t$ and $t_0$ except at the points where $\P_\pm^k=1$. In other words, $z^k=w^k$ and $t=t_0$. This is because 
\begin{equation}
\label{eq:Fund-123}
\begin{split} 
&\Gamma\Big(\frac{1-\lambda}{2}\Big)\Gamma\Big(\frac{1+\lambda}{2}\Big)\big(F^{+}_{\lambda,\ell}+F^{-}_{-\lambda, k-\ell}\big)\cr
=& \int_0^1\int_0^1 \Big\{ \Big(\frac{\varsigma^{-\frac{1-\lambda}{2}}(1-\varsigma+)^{-\frac{1-\lambda}{2}}\sigma^{-\frac{1+\lambda}{2}}
(1-\sigma)^{-\frac{1+\lambda}{2}}}{(1-\P_+^k\varsigma)(1-\P_-^k\sigma)}\cr
&\,\, \times \Big[ \frac{(\varsigma^{1/k}\P_+)^\ell \big(1-(\P_+\P_-(\varsigma\sigma)^{1/k}\big)^{k-\ell} }{1-(\P_+\P_-)^k\varsigma\sigma}\\
&\quad+
\frac{(\sigma^{1/k}\P_-)^{k-\ell }\big(1-(\P_+\P_-(\varsigma\sigma)^{1/k}\big)^\ell }{1-(\P_+\P_-)^k\varsigma\sigma}\Big]\Big\}
\frac{d\varsigma}{\varsigma}
\frac{d\sigma}{\sigma}. 
\end{split}
\end{equation}
The expression in braces is holomorphic in $\P+$ and $\P_-$ for $\P_+\P_-$ close to $1$, {\it i.e.,} $|\P|$ near $1$, and for $\varsigma,\, \sigma \in [0,1]$ because the possible zero of the denominator is cancelled by a zero in the numerator. This proves the argumenmt of the Remark.
%\medskip

Using (\ref{eq:Fund-123}), we have
\begin{equation}
\label{eq:cal-5}
\begin{split}
&K_+(z,w,t-t_0,\lambda)+K_+(\bar z,\bar w,t_0-t,-\lambda) \\
=& \, \frac{1}{4k\pi^2} (A)^{-\frac{1-\lambda}{2}} (\bar A)^{-\frac{1+\lambda}{2}}
\Big \{ \int_0^1 \frac{\big (1-|\P|^{2k} \varsigma\big)^{-\frac{1+\lambda}{2}}}{
\varsigma^{\frac{1+\lambda}{2}} (1-\varsigma)^{\frac{1-\lambda}{2}}}
\frac{d\varsigma}{\big ( 1-\P \varsigma^{\frac{1}{k}}\big )} \\
&\qquad\qquad \qquad\qquad \qquad+ \int_0^1 \frac{\big (1-|\P|^{2k} \varsigma\big)^{-\frac{1-\lambda}{2}}}{
\varsigma^{\frac{1-\lambda}{2}} (1-\varsigma)^{\frac{1+\lambda}{2}}} 
\frac{d\varsigma}{\big( 1-\bar\P \varsigma^{\frac{1}{k}}\big )}\Big\}.
\end{split}
\end{equation}
%Substituting the right-hand sides of~\eqref{eq:cal-4} and~\eqref{eq:cal-5} into~\eqref{eq:cal-3}, we finally obtain $K_\lambda$. 

Summarizing what we discussed above, we have the following formula for $K_{\lambda}$: 
\begin{equation}
\label{eq:Fund-125}
\begin{split}
&K_{\lambda}(z,w,t-t_0)\,\\
& = \frac{1}{8k\pi^2} (A)^{-\frac{1-\lambda}{2}} (\bar A)^{-\frac{1+\lambda}{2}}
\Big \{ \int_0^1 \frac{\big(1-|\P|^{2k}\varsigma\big)^{-\frac{1+\lambda}{2}}}{
\varsigma^{\frac{1+\lambda}{2}} (1-\varsigma)^{\frac{1-\lambda}{2}}}
\frac{1+\P \varsigma^{\frac{1}{k}}}{1-\P \varsigma^{\frac{1}{k}}}d\varsigma  \\
&\qquad+\int_0^1 \frac{\big (1-|\P|^{2k} \varsigma\big )^{-\frac{1+\lambda}{2}} }{
\varsigma^{\frac{1-\lambda}{2}}(1-\varsigma)^{\frac{1+\lambda}{2}}} 
\frac{1+\bar \P \varsigma^{\frac{1}{k}}}{1-\bar \P \varsigma^{\frac{1}{k}}}d\varsigma\Big \},
\end{split}
\end{equation}
for $-1<\re(\lambda)<1$. Here
\begin{equation}
\label{eq:func-A}
A\, =\,\frac 12\Big (|z|^{2k}+|w|^{2k}+i(t-t_0)\Big),
\end{equation}
and
\[
\P= \frac{\bar z\cdot w}{A^{\frac{1}{k}}},
\quad{\mbox {if}}\quad w\ne 0  \qquad {\mbox {and}}\qquad
\P=0, \quad{\mbox {if}}\quad w= 0. \qquad\blacksquare
\]
\medskip

\subsection{\bf Fundamental solution $K_{\lambda}$ for the Kohn Laplacian $\Delta_{\lambda}$ when $|\re(\lambda)|<1$}
Let us start with a well-known identity: 
\begin{equation}
\label{eq:basic-1}
\frac{1}{(1-w)^\beta}=F(1,\beta,1,w)=\frac{1}{\Gamma(\beta)\Gamma(1-\beta)} 
\int_0^1 \frac{\varsigma^{\beta-1}}{(1-\varsigma)^\beta(1-w\varsigma)}d\varsigma
\end{equation}
for $\beta\notin \Z$. Hence we have
\begin{equation*}
\begin{split}
&\Gamma\big (\frac{1+\lambda}{2}\big ) \Gamma\big (\frac{1-\lambda}{2}\big )
 \big (1-|\P|^{2k}\varsigma\big )^{-\frac{1+\lambda}{2}}\\
=&\, \int_0^1 \sigma^{-\frac{1-\lambda}{2}}
(1-\sigma)^{-\frac{1+\lambda}{2}}\big(1-|\P|^{2k}\varsigma\sigma
\big )^{-1}d\sigma.
\end{split}
\end{equation*}
It follows that
\begin{equation}
\label{eq:cal-6} \begin{split} &\int_0^1 \frac {
\varsigma^{-\frac{1+\lambda}{2}} (1-\varsigma)^{-\frac{1-\lambda}{2}}}{
\big (1-|\P|^{2k}\varsigma \big )^{-\frac{1+\lambda}{2}}\big
(1-|\P|^{2k}\varsigma \big )^{\frac{1+\lambda}{2}}}
\frac{1+\P \varsigma^{\frac{1}{k}}}{1- \P \varsigma^{\frac{1}{k}}} d\varsigma \\
=&\, \frac{1}{\Gamma\big (\frac{1+\lambda}{2}\big ) \Gamma\big(\frac{1-\lambda}{2}\big )}
\int_0^1\frac{d\sigma}{\sigma^{\frac{1-\lambda}{2}}
(1-\sigma)^{\frac{1+\lambda}{2}}}\\
&\,\,\times\int_0^1\frac{ \varsigma^{-\frac{1+\lambda}{2}}
(1-\varsigma)^{-\frac{1-\lambda}{2}}}{1-|\P|^{2k}\varsigma\sigma }
\frac{1+\P \varsigma^{\frac{1}{k}}}{1- \P \varsigma^{\frac{1}{k}}}d\varsigma.
\end{split}
\end{equation}
Similarly, we have
\begin{equation}
\label{eq:cal-7} 
\begin{split} &\int_0^1 \frac{
\varsigma^{-\frac{1-\lambda}{2}} (1-\varsigma)^{-\frac{1+\lambda}{2}}}{
\big (1-|\P|^{2k}\varsigma \big )^{-\frac{1-\lambda}{2}}\big
(1-|\P|^{2k}\varsigma \big )^{\frac{1+\lambda}{2}}} \cdot
\frac{1+\bar \P \varsigma^{\frac{1}{k}}}{1- \bar \P \varsigma^\frac{1}{k}} d\varsigma \\
=&\, \frac{1}{\Gamma\big (\frac{1+\lambda}{2}\big ) \Gamma\big
(\frac{1-\lambda}{2}\big )}
\int_0^1\frac{d\sigma}{\sigma^{\frac{1-\lambda}{2}}
(1-\sigma)^{\frac{1+\lambda}{2}}}\\
&\quad\times \int_0^1\frac{ \varsigma^{-\frac{1+\lambda}{2}}
(1-\varsigma)^{-\frac{1-\lambda}{2}}}{1-|\P|^{2k}\varsigma\sigma }
\frac{1+\bar \P \varsigma^\frac{1}{k}}{1- \bar \P \varsigma^\frac{1}{k}}d\varsigma.
\end{split}
\end{equation}
Thus, the sum of~\eqref{eq:cal-6} and~\eqref{eq:cal-7} gives us
\begin{equation*}
\begin{split}
& \frac{2}{\Gamma\big (\frac{1+\lambda}{2}\big ) \Gamma\big
(\frac{1-\lambda}{2}\big )} \int_0^1\int_0^1
\varsigma^{-\frac{1+\lambda}{2}} (1-\varsigma)^{-\frac{1-\lambda}{2}} 
\sigma^{-\frac{1-\lambda}{2}}
(1-\sigma)^{-\frac{1+\lambda}{2}}\\
&\,\, \times \prod_{\ell=1}^{k-1} 
\Big (1-e^{\frac{2\ell\pi}{k}i}|\P|^2(\varsigma\sigma)^\frac{1}{k}\Big )^{-1} 
\frac{d\varsigma d\sigma}{(1-\P \varsigma^\frac{1}{k})(1-\bar \P \sigma^\frac{1}{k})}.
\end{split}
\end{equation*}

From the above and formula (\ref{eq:Fund-125}), we obtain for $k>1$,
\begin{equation}
\label{eq:cal-8}
\begin{split}
K_{\lambda} =&\, \frac{1}{4k\pi^2} 
\frac{(A)^{-\frac{1-\lambda}{2}}(\bar A)^{-\frac{1+\lambda}{2}} }{
\Gamma\big (\frac{1+\lambda}{2}\big ) \Gamma\big (\frac{1-\lambda}{2}\big )} 
\int_0^1 \frac{\varsigma^{-\frac{1+\lambda}{2}} (1-\varsigma)^{-\frac{1-\lambda}{2}}}{
1-\P \varsigma^\frac{1}{k}}d\varsigma\\
&\,\,\times \int_0^1 \prod_{\ell=1}^{k-1} \Big
(1-e^{\frac{2\ell\pi}{k}i}|\P|^2(\varsigma\sigma)^\frac{1}{k}\Big
)^{-1}\, \frac{\sigma^{-\frac{1-\lambda}{2}}
(1-\sigma)^{-\frac{1+\lambda}{2}}} {1-\bar \P \sigma^{\frac{1}{k}}}d\sigma,
\end{split}
\end{equation}
and for $k=1$,
\begin{equation}
\label{eq:cal-9}
K_{\lambda} =\, \frac{1}{4\pi^2} 
\frac{(A)^{-\frac{1-\lambda}{2}}(\bar A)^{-\frac{1+\lambda}{2}} }{
\Gamma\big (\frac{1+\lambda}{2}\big ) \Gamma\big (\frac{1-\lambda}{2}\big )} 
\int_0^1 \frac{\varsigma^{-\frac{1+\lambda}{2}} (1-\varsigma)^{-\frac{1-\lambda}{2}}}{
1-\P \varsigma}d\varsigma
\int_0^1  \frac{\sigma^{-\frac{1-\lambda}{2}}
(1-\sigma)^{-\frac{1+\lambda}{2}}}{1-\bar \P \sigma}d\sigma
\end{equation}

Now we may state our result as following theorem.
\medskip
\begin{theorem}
\label{thm:fund-100} Assume $(z,t)\ne(w,t_0)$. Then the
fundamental solution $K_{\lambda}$ for the sub-Laplacian
$\Delta_{\lambda}$ is given by~\eqref{eq:cal-8}. In particular,
when $k=1$, $K_{\lambda}$ is given by~\eqref{eq:cal-9}.
\end{theorem}
\smallskip

\noindent {\bf Remarks.} 
\smallskip

\noindent $(1)$. In fact, from what we have discussed, we know that $|\P|\le 1$ and $\P=1$ if and only if $(z,t)=(w,t_0)$.
Therefore, it is easy to see that $K_{\lambda}$ has a unique
singularity at $(w,t_0)$. Moreover, it is not difficult to show
that $K_{\lambda} \in L^1_{loc}(\R^3)$ and $\Delta_{\lambda}
K_{\lambda}=\delta(z,w,t-t_0)$. We omit the detail here.
\smallskip

\noindent $(2)$. The integrand in (\ref{eq:cal-8}) is real analytic in $\P$ and $\bar \P$ in the region
$$
\mathcal W\,=\, \Big \{ |\P|\le 1,\,\,  \, \P \not=1, \, \bar \P \not=1, \Big\}
$$
when $\varsigma,\, \sigma \in [0,1]$. Moreover, 
\begin{equation*}
%\begin{split} 
F_{\lambda}(0,0)\,=\, \frac{\mathcal B\big(\frac{1-\lambda}{2}, \frac{1+\lambda}{2}\big)\mathcal B\big(\frac{1+\lambda}{2}, \frac{1-\lambda}{2}\big)}
{\Gamma\big (\frac{1-\lambda}{2}\big )\Gamma\big (\frac{1+\lambda}{2}\big )}
\,=\, \Gamma\big (\frac{1-\lambda}{2}\big )\Gamma\big (\frac{1+\lambda}{2}\big ),
%\end{split}
\end{equation*} 
where $\mathcal B(\cdot,\cdot )$ denotes the beta function. 
\smallskip 

From Theorem \ref{thm:fund-100}, we obtain the fundamental solution for the operator $\Delta_{0}$ as a corollary.
\begin{corollary}
\label{cor:fund-101} Assume $(z,t)\ne(w,t_0)$. Then the
fundamental solution $K_{0}$ for the sub-Laplacian
$\Delta_{0,k}$ has the following closed form:
\begin{equation}
\label{eq:cal-10}
\begin{split}
K_{0} =&\, \frac{1}{4k\pi^3|A|} \int_0^1\int_0^1
\varsigma^{-\half} (1-\varsigma)^{-\half} \sigma^{-\half}
(1-\sigma)^{-\half}\\
&\,\, \times \prod_{\ell=1}^{k-1} \Big (1-e^{\frac{2\ell\pi}{k}i}|\P|^2(\varsigma\sigma)^\frac{1}{k}\Big )^{-1} 
\frac{d\varsigma d\sigma}{(1-\P \varsigma^\frac{1}{k})(1-\bar \P \sigma^\frac{1}{k})}.
\end{split}
\end{equation}
\end{corollary}
\medskip

\section{\bf The Cauchy--Szeg\"o kernel on $\Omega_k$ and $\partial\Omega_k$}

Let $f(\bz,z_{n+1}) \in H^2(\Omega_k)$. Here $H^2(\Omega_k)$ is the space of all square integrable
holomorphic functions on $\Omega_k$. Then one has
\[
f(\bz,z_{n+1})=\int_0^\infty e^{2\pi i\lambda z_{n+1}}\widetilde  f(\bz, \lambda)
d\lambda
\]
where $\widetilde  f(\bz,\lambda)$ is the Fourier transform of $f$
with respect to the $z_{n+1}$ variable. The ``height function" on
$\Omega_k$ can be written as
\[
\rho=\im(z_{n+1})-\left (\sum_{j=1}^n|z_j|^2\right)^k=
\im(z_{n+1})-|\bz|^{2k},
\]
where $|\bz|^{2}=\sum_{j=1}^n|z_j|^2$.
 Then
\begin{equation}
\label{eq:1} f(\bz,z_{n+1})=\int_0^\infty e^{2\pi
i\lambda(\re(z_{n+1}))-2\pi\lambda (\rho+|\bz|^{2k})}\widetilde
f(\bz, \lambda) d\lambda.
\end{equation}
According to Plancherel's formula, we know that
\[
e^{-2\pi\lambda (\rho+|\bz|^{2k})}\widetilde  f(\bz, \lambda)=
\int_{-\infty}^{+\infty} e^{2\pi
i\lambda(\re(z_{n+1}))}f(\bz,z_{n+1})d(\re (z_{n+1})).
\]
This implies that
\begin{equation}
\label{eq:2} \widetilde  f(\bz,\lambda)=\int_{-\infty}^{+\infty}
e^{-2\pi i\lambda \bar z_{n+1}}f(\bz,z_{n+1})d(\re (z_{n+1})).
\end{equation}
The definition of $H^2(\Omega_k)$ and~\eqref{eq:1} imply that
\[
\int_{\C^n}\int_0^\infty e^{-4\pi\lambda \left (\sum_{j=1}^n|z_j|^2\right)^k}
|\widetilde  f(\bz, \lambda)|^2
d\lambda dV(\bz)<\infty,
\]
where $dV(\bz)=dz_1 d\bar z_1\cdots dz_n d\bar z_n$. Let
$A_\lambda$ be the Bergman space of holomorphic functions in
$\C^n$ with the norm
\[
\|g\|_{A_\lambda(\Omega_k)}=\left \{ \int_{\C^n} e^{-4\pi
\lambda|\bz|^{2k}}|g(\bz)|^2 dV(\bz)\right\}^{1\over 2}<\infty.
\]
Then $A_\lambda$ has a reproducing kernel $S_\lambda(\bz,\bw)$, {\it i.e.,}
for $g\in A_\lambda$, one has
\[
g(\bz)=\int_{\C^n} S_\lambda(\bz,\bw)dV(\bw).
\]
In particular,
\[
\widetilde  f(\bz,\lambda)=\int_{\C^n} S_\lambda(\bz,\bw)\widetilde  f(\bw,\lambda) dV(\bw),
\]
whenever $f\in H^2(\Omega_k)$. Therefore, such an $f$ has the following
integral representation:
\[
f(\bz,z_{n+1})=\int_0^\infty e^{2\pi i\lambda z_{n+1}}d\lambda
\int_{\C^n} S_\lambda(\bz,\bw)\widetilde  f(\bw,\lambda) dV(\bw).
\]
Substituting~\eqref{eq:2} for $\widetilde  f(\bw,\lambda)$, we
obtain
\[
f(\bz,z_{n+1})=\int_{\partial \Omega_k} S(\bz,z_{n+1};\bw, w_{n+1})
f(\bw,w_{n+1}) dV(\bw) d(\re (w_{n+1})).
\]
Denote
\[
A={i\over 2}(\bar w_{n+1}-z_{n+1}),
\]
and
\[
S(\bz,z_{n+1};\bw, w_{n+1})=\int_0^\infty e^{-4\pi \lambda A} S_\lambda(\bz,\bw)d\lambda.
\]
Then we have the following lemma.
\begin{lemma} Let $U$ be a unitary transform on $\C^n$. If $S(\bz,z_{n+1};\bw, w_{n+1})$
is the Cauchy--Szeg\"o kernel on $\Omega_k$, then
\[
S(U(\bz),z_{n+1};U(\bw), w_{n+1})=S(\bz,z_{n+1};\bw, w_{n+1}).
\]
\end{lemma}
\medskip
\begin{proof} Let $f\in H^2(\Omega_k)$. Then $f\circ U^{-1}\in H^2(\Omega_k)$
for every unitary transform $U$ and
\begin{eqnarray*}
&&\int_{\partial \Omega_k}S(U(\bz),z_{n+1};U(\bw), w_{n+1}) f(\bw,w_{n+1})
dV(\bw) d(\re (w_{n+1}))\\
&&= \int_{\partial \Omega_k}S(U(\bz),z_{n+1};\bw, w_{n+1}) f(U^{-1}(\bw),w_{n+1})
dV(\bw) d(\re (w_{n+1}))\\
&&=f(U^{-1}(\cdot),z_{n+1})\big|_{U(\bz)}
= f(\bz,z_{n+1})\\
&&= \int_{\partial \Omega_k}S(\bz,z_{n+1};\bw, w_{n+1}) f(\bw,w_{n+1})
dV(\bw) d(\re (w_{n+1})).
\end{eqnarray*}
The lemma follows immediately.
\end{proof}

\begin{theorem}\label{thm CS kernel}
The explicit Cauchy--Szeg\"o kernel on $\Omega_k$ is as follows. 
\begin{eqnarray*}
&&S(\bz,z_{n+1};\bw, w_{n+1})\\
&&={n!\over 4\pi^{n+1}}{1 \over \left (A^{1\over k}-
\sum_{j=1}^n z_j\bar w_j\right)^{n+1}A^{k-1\over k} }\\
&&= {n!\over 4\pi^{n+1}} {\left [{i\over 2}(s-t)+{1\over 2}(|\bz|^{2k}+|\bw|^{2k})
+{1\over 2}(\rho+\mu)\right ]^{1-k\over k}\over
\left \{ \left [ {1\over 2}(|\bz|^{2k}+|\bw|^{2k})-{i\over 2}(t-s)+ {1\over 2} (\rho+\mu)\right ]^{1\over k}- \sum_{j=1}^n z_j\bar w_j\right \}^{n+1}}.
\end{eqnarray*}
\end{theorem}

\medskip
Let us first compute $S_\lambda(\bz,\bw)$ as follows: Assume that
$\bw={{\bf 1}}=(1,0,\dots, 0)$ is the ``north pole" of the unit sphere
in $\C^n$, then we have
\[
S_\lambda(\bz, {\bf 1})=\sum_{j=0}^\infty {\bz^j{\bf 1}^j\over
\| z_1^j\|_{A_\lambda(\Omega_k)}^2} =\sum_{j=0}^\infty {z_1^j\over
\| z_1^j\|_{A_\lambda(\Omega_k)}^2}
\]
where
\begin{eqnarray*}
\| z_1^j\|_{A_\lambda(\Omega_k)}^2&=&\int_{\C^n} e^{-4\pi
\lambda|\bz|^{2k}}
|z_1|^{2j}dv(\bz)\\
&=&\int_{\partial B_n}|z_1|^{2j} d\sigma(\bz) \int_0^\infty
e^{-4\pi \lambda r^{2k}} r^{2j}r^{2n-1}dr
\end{eqnarray*}
with $r=|\bz|=\left (\sum_{\ell=1}^n|z_\ell|^2\right )^{1\over
2}$. Set $u=4\pi\lambda r^{2k}$. It follows that
\[
{du\over u}=2k{dr\over r}.
\]
Hence we have
\begin{eqnarray*}
\|z_1^j\|_{A_\lambda(\Omega_k)}^2&=&{\pi^n2(j!)\over (j+n-1)!}
\cdot {1\over 2k(4\pi \lambda)^{j+n\over k}}\int_0^\infty
e^{-u}u^{j+n\over k} {du\over u}  \\
&=& {\pi^n 2(j!)\over (j+n-1)!}
\cdot {1\over 2k(4\pi \lambda)^{j+n\over k}} \Gamma\left({j+n\over k}\right).
\end{eqnarray*}
Therefore,
\[
S_\lambda(\bz, {\bf 1})={k\over \pi^n} \sum_{j=0}^\infty {(j+n-1)!\over j!}
{z_i^j\over \Gamma\left({j+n\over k}\right)} (4\pi \lambda)^{j+n\over k}.
\]
Hence
\begin{eqnarray*}
S(\bz,z_{n+1};{\bf 1}, w_{n+1})&=& {k\over \pi^n} \sum_{j=0}^\infty {(j+n-1)!\over j!}
\left (\int_0^\infty e^{-4\pi\lambda A}(4\pi\lambda)^{j+n\over k} d\lambda
\right ){z_1^j\over\Gamma\left({j+n\over k}\right)} \\
&=& {k\over 4\pi^{n+1}}
\sum_{j=0}^\infty {(j+n-1)!\over j!}
{\Gamma\left ({j+n\over k}+1\right)\over \Gamma\left ({j+n\over k}\right)}
A^{-{j+n\over k}-1}z_1^j\\
&=& {1\over 4\pi^{n+1}}
\sum_{j=0}^\infty {(j+n)!\over j!}\left ({z_1\over A^{1\over k}}\right )^j A^{-{n\over k}-1}\\
&=& {n!\over 4\pi^{n+1}}
\sum_{j=0}^\infty {(j+n)!\over j!n!}\left ({z_1\over A^{1\over k}}\right )^j A^{-{n\over k}-1}\\
&=& {n!\over 4\pi^{n+1}} \left (1-{z_1\over A^{1\over k}}\right )^{-n-1}
A^{-{n\over k}-1}.
\end{eqnarray*}
Suppose $\bz=r\cdot \bz^\prime$ with $\|\bz^\prime\|=1$
and $\bz^\prime\in \partial B_n$. Then there exists a unitary transform
$U$ on $\partial B_n$ such that $U(\bz^\prime)={\bf 1}$. Hence,
$\bz^\prime=U^{-1}({\bf 1})$. Therefore,
\begin{eqnarray*}
S(\bz,z_{n+1};\bw, w_{n+1})&=& S(r\cdot \bz^\prime,z_{n+1};\bw, w_{n+1})
= S(r\cdot U^{-1}({\bf 1}),z_{n+1};\bw, w_{n+1})\\
&=& S(r\cdot {\bf 1},z_{n+1};U(\bw), w_{n+1})
=\overline {S(U(\bw), w_{n+1}; r\cdot {\bf 1},z_{n+1})}\\
&=& {n!\over 4\pi^{n+1}}{1\over \left (1-r\cdot \left({\overline {U(\bw)}\over
A^{1\over k}}\right)_1\right)^{n+1}}A^{-{n\over k}-1}\\
&=& {n!\over 4\pi^{n+1}}{1\over \left (1-(r{\bf 1}\cdot
\left({\overline {U(\bw)}\over A^{1\over k}}\right )\right)^{n+1}
A^{{n\over k}+1}}\\
&=& {n!\over 4\pi^{n+1}}{1\over \left (1-(rU^{-1}({\bf 1})\cdot
\left({\bar\bw\over A^{1\over k}}\right)\right)^{n+1}A^{{n\over k}+1}}\\
&=& {n!\over 4\pi^{n+1}}{1\over \left (1-{\bz\cdot \bar \bw \over
A^{1\over k}}\right)^{n+1}A^{{n\over k}+1}}
= {n!\over 4\pi^{n+1}}{A^{{1\over k}-1} \over \left (A^{1\over k}-
\sum_{j=1}^n z_j\bar w_j\right)^{n+1}}.
\end{eqnarray*}
This tells us that
\begin{eqnarray*}
&&S(\bz,z_{n+1};\bw, w_{n+1})\\
&&={n!\over 4\pi^{n+1}}{1 \over \left (A^{1\over k}-
\sum_{j=1}^n z_j\bar w_j\right)^{n+1}A^{k-1\over k} }\\
&&= {n!\over 4\pi^{n+1}} {\left [{i\over 2}(s-t)+{1\over 2}(|\bz|^{2k}+|\bw|^{2k})
+{1\over 2}(\rho+\mu)\right ]^{1-k\over k}\over
\left \{ \left [ {1\over 2}(|\bz|^{2k}+|\bw|^{2k})-{i\over 2}(t-s)+ {1\over 2} (\rho+\mu)\right ]^{1\over k}- \sum_{j=1}^n z_j\bar w_j\right \}^{n+1}}.
\end{eqnarray*}
When both $(\bz,z_{n+1})$ and $(\bw, w_{n+1})$ in $\partial \Omega_k$, {\it $\rho=\mu=0$}, then we have
\[
S(\bz,t;\bw, s)= {n!\over 4\pi^{n+1}} {\left [{i\over 2}(s-t)+{1\over 2}(|\bz|^{2k}+|\bw|^{2k})
\right ]^{1-k\over k}\over
\left \{ \left [ {1\over 2}(|\bz|^{2k}+|\bw|^{2k})-{i\over 2}(t-s)\right ]^{1\over k}- \sum_{j=1}^n z_j\bar w_j\right \}^{n+1}}.
\]
In particular, when $k=1$, the Cauchy--Szeg\"o for the Heisenberg
group $\H_n$ is
\[
S(\bz,t;\bw, s)={n!\over 4\pi^{n+1}} {1\over
\left \{ \left [ {1\over 2}(|\bz|^{2k}+|\bw|^{2k})-{i\over 2}(t-s)\right ]^{1\over k}- \sum_{j=1}^n z_j\bar w_j\right \}^{n+1}}.
\]
We note that $ | [ {1\over 2}(|\bz|^{2k}+|\bw|^{2k})-{i\over 2}(t-s)]^{1\over k}- \sum_{j=1}^n z_j\bar w_j|^{\frac 12}$ can be considered as a generalization of the Kor\'anyi distance to the higher step operators (see Diaz \cite{Di}). 
\smallskip

\begin{remark} 
 
\noindent $(1)$. The domain $\Omega_k$ is equivalent to the ``ellipsoid"
\[
\mathbb{E}_k=\Big\{ (\bw,w_{n+1})\in \C^{n+1}:\, \big
(\sum_{j=1}^n |w_j|^2\big)^k+|w_{n+1}|^2<1\Big\}
\]
via the ``generalized Caley transform":
\[
z_1={w_1\over (1+w_{n+1})^{1\over k}},\dots,
z_n={w_n\over (1+w_{n+1})^{1\over k}},
z_{n+1}={i(1-w_{n+1})\over 1+w_{n+1}}.
\]
Then the Cauchy--Szeg\"o kernel for $\mathbb{E}_k$ is
\[
S(\bz,z_{n+1};\bw, w_{n+1})={n!\over 4\pi^{n+1}}
{1\over \left [ (1-z_{n+1}\bar w_{n+1})^{1\over k}-
\sum_{j=1}^n z_j\bar w_j\right ]^{n+1}(1-z_{n+1}\bar w_{n+1})^{k-1\over k}}.
\]
In particular, when $k=1$, then the Cauchy--Szeg\"o kernel for the
unit ball $\mathbb{B}_{n+1}\subset \C^{n+1}$ is
\[
S(\bz,z_{n+1};\bw, w_{n+1})\,=\, {n!\over 4\pi^{n+1}}\frac{1}{\big(1-
\sum_{j=1}^{n+1} z_j\bar w_j\big)^{n+1}}.
\]

\noindent $(2)$. When $k=1$, then the Cauchy--Szeg\"o projection is closely related to the solvability of the Kohn Laplacian. 
Let us consider the Kohn Laplacian acting on functions, {\it i.e.,} $q=0$. Then we know that 
\[
\square_b\,=\, -\frac{1}{2}\sum_{k=1}^{n}\big(Z_k\bar Z_k+\bar Z_kZ_k\big)-i [Z_k,\bar Z_k]\,=\, -\sum_{k=1}^n Z_k\bar Z_k.
\]
In this case, the operator annihilates the boundary values of holomorphic functions on ${\bf H}_n$. Now we need to deal with the Hans Lewy operator. 
In general, we can't expect this operator is hypoelliptic. 
Moreover, the equation 
\[
\sum_{k=1}^n Z_k\bar Z_k(u)\,=\, f
\]
is generally not even locally solvable. Following a method in Greiner--Kohn--Stein \cite{GrJS} and Greiner--Stein \cite{GrS1}, we know that for any $f\in L^2({\bf H}_n)$ leads to the 
Cauchy--Szeg\"o integral ${\mathcal C}(f)$, defined in $C^\infty_0({\bf H}_n)$ by 
\[
{\mathcal C}(f)(\bz,z_{n+1};\bw, w_{n+1})\,=\, \int_{\partial \widetilde\Omega_n} S(\bz,z_{n+1};\bw, w_{n+1}) f(\bw,w_{n+1})d\sigma,
\]
with 
\[
S(\bz,z_{n+1};\bw, w_{n+1})={2^{n-1} n!\over \pi^{n+1}}\Big\{i(\bar w_{n+1}-z_{n+1}) -2\sum_{k=1}^n z_k\bar w_k \Big\}^{-n-1}.
\]
Here $d\sigma$ is the Lebesgue measure defined on $\partial \widetilde \Omega_n$ which is identified as the Heisenberg group ${\bf H}_n$. The restriction 
of ${\mathcal C}(f)$ to $\partial \widetilde \Omega_n$ is given by 
\begin{equation}
\label{eq:C-S-1}
{\mathcal C}_b(f)\,=\, \lim_{\rho\to 0^+} f\ast S_\rho,
\end{equation} 
where 
\[
S_\rho(\bz,t)={2^{n-1} n!\over \pi^{n+1}}\Big(\rho^2+\sum_{k=1}^n |z_k|^2-it \Big)^{-n-1}.
\]
The convolution in~\eqref{eq:C-S-1} is with respect to the Heisenberg group. Since $f\in L^2$, the limit in~\eqref{eq:C-S-1} exists in $L^2$-norm 
(see Kor\'anyi and V\'agi \cite{KV}). 

Let us consider the following equation: 
\[
\square_b^{(0)}\,=\, \mathcal L_\lambda-i\big(\lambda-n\big) \frac{\partial}{\partial t}.
\]
Thus 
\begin{equation}
\label{eq:C-S-2}
\square_b^{(0)} \big(K_\lambda \big) \,=\, {\mathcal L}_\lambda\big(K_\lambda \big) -i\big(\lambda-n\big) \frac{\partial}{\partial t}\big(K_\lambda \big),
\end{equation}
where 
\[
K_\lambda(\bz, t)\,=\, \frac {2^{2-n}\pi^{n+1}}{\Gamma \big(\frac{n+\lambda}{2}\big) \Gamma \big(\frac{n+\lambda}{2}\big) }
\Big(\sum_{k=1}^n|z_k|^2-it\Big)^{-\frac{n+\lambda}{2}}\Big(\sum_{k=1}^n|z_k|^2+it\Big)^{-\frac{n-\lambda}{2}}.
\]
Formal differentiation of~\eqref{eq:C-S-2} with respect to the variable $\lambda$ yields the following result
\begin{equation*}
\begin{split}
&\square_b^{(0)}\Big[ \frac {2^{n-2}(n-1)!}{\pi^{n+1}}\log\Big( \frac {|\bz|^2-it}{|\bz|^2+it}\Big) \cdot  \Big(\sum_{k=1}^n|z_k|^2-it\Big)^{-n}\Big]\\
=&\,\,\, 
\delta- \frac {2^{n-1}n!} {\pi^{n+1}} \Big(\sum_{k=1}^n|z_k|^2-it\Big)^{-(n+1)} .
\end{split}
\end{equation*}
Denote 
\[
\Psi\,=\, \frac {2^{n-2}(n-1)!}{\pi^{n+1}}\log\big( \frac {|\bz|^2-it}{|\bz|^2+it}\Big) \cdot \Big(\sum_{k=1}^n|z_k|^2-it\Big)^{-n}.
\] 
Then we have the following identity 
\[
\square_b^{(0)}K\,=\, K\square_b^{(0)}\,=\, {\bf I}-\mathcal C_b,
\]
where $K(f)\,=\, f\ast \Psi$ with $f\in C^\infty_0({\bf H}_n)$. 
\end{remark}
\medskip

\section{\bf Sharp Estimates of Cauchy--Szeg\"o kernel} 

\color{black}
Consider the triple $(\partial\Omega_k, d,\mu)$, where $\mu$ is the Lebesgue measure on $\mathbb C\times \mathbb R$.
We use new coordinates $(\bz,t)$
on the boundary $\partial\Omega_k$ to identify it as
$\C\times\R$. Here $\bz=x+iy$ and
$t=\re(z_{2})$.

\color{black}

We now introduce the quasi-distance on $\partial\Omega_k$ as follows: for every $ (\bz, t),(\bw, s) \in \partial \Omega_k$,
\begin{align}\label{dist}
d(\color{black}{ (\bz, t),(\bw, s) }):=h^2(\color{black}{ (\bz, t),(\bw, s) })\rho^{2-2k}(\color{black}{ (\bz, t),(\bw, s) }),
\end{align}
where
$$\rho(\color{black}{ (\bz, t),(\bw, s) ):=  |\bz|+|\bw|+|\sigma|^{1\over2k}\approx  |\bz|+|\bw|+|t-s|^{1\over2k}
},$$
and
$$h(\color{black}{ (\bz, t),(\bw, s) })=|\bz-\bw|^2\rho^{2k-2}(\color{black}{ (\bz, t),(\bw, s) })+|\sigma(\color{black}{ (\bz, t),(\bw, s) })|$$
with
$$ \sigma( (\bz, t),(\bw, s) ) =  t-s+2\im (\bz^k\overline{\bw}^k).$$
Based on Proposition 9.6 in \cite{BGGr3}, we see that this quasi-metric $d$ satisfies the quasi-triangle inequality:
$$ d((\bz, t),(\bu, r)) \lesssim d((\bz, t),(\bw, s))+ d((\bw, s),(\bu, r)).$$

\color{black}
{Recall the functions in \eqref {eq:func-A} and \eqref{P}}
\begin{align*}
&A(\bz,t;\bw, s)= {1\over 2}\big(|\bz|^{2k}+|\bw|^{2k} -i(t-s)\big);\\[6pt]
%&A_-(\bz,t;\bw, s)=\overline{A(\bz,t;\bw, s)};\\[6pt]
&\P(\bz,t;\bw, s)={{\bz \overline{\bw}\over A(\bz,t;\bw, s)^{1\over k}}}.\\[6pt]
%&P_-(\bz,t;\bw, s)=\overline{P(\bz,t;\bw, s)}.
\end{align*}
By Lemma 9.3 in \cite{BGGr3} and the estimate on Page 242 in \cite{BGGr4}, we have the following properties.
\begin{lemma}\label{lem6.1}
The functions $h, \rho, A$ and $\P$ satisfy
\begin{equation}\begin{split}\label{est 1}
& |\bz^k-\bw^k|^2\lesssim h ((\bz, t),(\bw, s))\lesssim \rho^{2k} ((\bz, t),(\bw, s)) \approx |A(\bz,t;\bw, s)|; \\[6pt]
& | 1-P(\bz,t;\bw, s) |\approx {h ((\bz, t),(\bw, s))\over |A(\bz,t;\bw, s)|},\quad (\bz,t), (\bw, s)\in\partial\Omega_k. 
\end{split}\end{equation}
\end{lemma}
Based on these notation, we see that the Cauchy-Szeg\"o  kernel  $S(\bz,t;\bw, s)$ on $\partial \Omega_k$ can be expressed as 
\begin{align}
S(\bz,t;\bw, s) = {1\over 4\pi^2}A^{-{k+1\over k}}(\bz,t;\bw, s) \big(1-\P(\bz,t;\bw, s)\big)^{-2}.
\end{align}
Moreover, based on \eqref{est 1}, we see that
\begin{align}\label{est33}
d((\bz,t),(\bw, s)) \lesssim \rho^{2k+2} ((\bz, t),(\bw, s)).
\end{align}
%\medskip

%\medskip

\color{black}

\begin{theorem}\label{thm 7.1}
For both $(\bz,t)$ and $(\bw, s)$ in $\partial \Omega_k$ with $(\bz,t)\not=(\bw, s)$,  the Cauchy--Szeg\"o projection associated with the kernel  $S(\bz,t;\bw, s)$ is a Calder\'on-Zygmund operator on $(\partial\Omega_k, d,\mu)$.
\end{theorem}
\begin{proof}

We first consider the size estimate of $S(\bz,t;\bw, s)$. Note that for $(\bz,t)\not=(\bw, s)$, by Lemma \ref{lem6.1}, we have
\allowdisplaybreaks{
\begin{align}\label{size}
|S(\bz,t;\bw, s)|&={1\over 4\pi^2} \left| A^{-{k+1\over k}}(\bz,t;\bw, s) \big(1-\P(\bz,t;\bw, s)\big)^{-2}\right|\nonumber\\[6pt]
&\approx | A(\bz,t;\bw, s)|^{-{k+1\over k}}  | A(\bz,t;\bw, s)|^{2} h^{-2} ((\bz, t),(\bw, s))\nonumber\\[6pt]
&= | A(\bz,t;\bw, s)|^{-{1\over k}+1}   h^{-2} ((\bz, t),(\bw, s))\\[6pt]
&\approx \rho^{-(2-2k)} ((\bz, t),(\bw, s))  h^{-2} ((\bz, t),(\bw, s))\nonumber\\[6pt]
&={1\over d((\bz,t),(\bw, s))}. \nonumber
\end{align}
}

Next, we verify the regularity conditions. Consider the difference
$$ |S(\bz,t;\bw_1, s_1)-S(\bz,t;\bw_0, s_0)|, $$
where 
\color{black}
\begin{align}\label{d10}
 d((\bw_1, s_1),  (\bw_0, s_0) )\leq \color{black}{c} d((\bz,t),(\bw_0, s_0)),
 \end{align}
for some small $c$.
 \color{black}
 
To continue, we consider the following substitution: fix $(\bz,t)$, 
\color{black} 
by the change of variables
\color{black}
\begin{align}\label{zw}
    \left\{
                \begin{array}{ll}
                  \bw'=\bw,\\[6pt]
                  s'=s-2\im (\bz^k \overline{\bw}^k),
                \end{array}
              \right.
  \end{align}
\color{black}
we have  
$$S(\bz,t;\bw, s)=S(\bz,t;\bw', s'+2\im(z^k\overline{\bw'}^k))=:\widetilde{S}(\bz,t;\bw', s'),$$
and thus,
\smallskip  

\begin{align}\label{partial}
    \left\{
                \begin{array}{ll}
                 {\displaystyle\partial\widetilde{S}\over\displaystyle \partial\bw'} ={\displaystyle\partial S\over\displaystyle \partial\bw} + i k\bw^{k-1} \overline{\bz}^k {\displaystyle\partial S\over\displaystyle \partial s},\\[11pt]
                   {\displaystyle\partial \widetilde{S}\over\displaystyle \partial s'}= {\displaystyle\partial S\over\displaystyle \partial s}.\\
                \end{array}
              \right.
  \end{align}

 \color{black}
 
% \begin{align}\label{partial}
%    \left\{
%                \begin{array}{ll}
%                 {\displaystyle\partial\over\displaystyle \partial\bw'} ={\displaystyle\partial\over\displaystyle \partial\bw} + i k\bw^{k-1} \overline{\bz}^k {\displaystyle\partial \over\displaystyle \partial s},\\[11pt]
%                   {\displaystyle\partial \over\displaystyle \partial s'}= {\displaystyle\partial \over\displaystyle \partial s}.\\
%                \end{array}
%              \right.
%  \end{align}

 \color{black}

Let 
\begin{align*}
s'_\alpha&=s'(\bw_\alpha, s_\alpha), &\alpha=0,1,\\
(\bw'_\nu, s'_\nu)&=(1-\nu)(\bw_0,s'_0)+\nu (\bw_1,s'_1), &0\leq\nu\leq 1.
\end{align*}
Let $(\bw_\nu, s_\nu)$ denote the point whose $(\bw',s')$ coordinates are $(\bw'_\nu, s'_\nu)$. Therefore,
\begin{align*}
\bw_\nu&=(1-\nu)\bw_0+\nu \bw_1,\\
s_\nu&=s'_\nu+2\im(\bz^k\overline{\bw}^k_{\nu})\\
&=(1-\nu)\big(s_0-2\im(\bz^k\overline{\bw}^k_{0})\big)+\nu\big(s_1-2\im(\bz^k\overline{\bw}^k_{1})\big)\\
&\quad+2\im\left[ \bz^k\left( (1-\nu)\overline{\bw}_0+\nu\overline{\bw}_k\right)^k\right].
\end{align*}

%%%%%%%%%%%%%%%%%%%%%%%%%%%%%%%%%%%
%For the functions $\rho, h, d$, we use the notation
%\begin{align*}
%\rho_\nu&=\rho((z,t), (w_\nu,s_\nu)), \quad h_\nu=h((z,t), (w_\nu,s_\nu)),\quad d_\nu=d((z,t), (w_\nu,s_\nu)),\\
%\rho_{\mu\nu}&=\rho((w_\mu,s_\mu), (w_\nu,s_\nu)), \quad h_{\mu\nu}=h((w_\mu,s_\mu), (w_\nu,s_\nu)),\quad d_{\mu\nu}=d((w_\mu,s_\mu), (w_\nu,s_\nu)).
%\end{align*}
%By \cite[(9.34), (9.35), (9.37)]{BGGr3}, we have
%\begin{align}\label{est11}
%&\rho_\nu\approx\rho_0,
%\quad h_\nu\approx h_0,\quad h_{10}\lesssim h_0,\quad 0\leq\nu\leq 1,\\
% &|w_1-w_0|\lesssim d^{1\over 4}_{10}\rho^{\frac{1-k}{2}}_{10},\quad
%|s'_1-s'_0|\leq h_{10}+h^{1\over 2}_0 h^{1\over 2}_{10}.\nonumber
%\end{align}
%%%%%%%%%%%%%%%%%%%%%%%%%%%%%%%

By \cite[(9.34), (9.35), (9.37)]{BGGr3}, we have
\begin{equation}\begin{split}\label{est11}
\rho((\bz,t), (\bw_\nu,s_\nu))&\approx\rho((\bz,t), (\bw_0,s_0)),\\
\quad h((\bz,t), (\bw_\nu,s_\nu))&\approx h((\bz,t), (\bw_0,s_0)),\quad 0\leq\nu\leq 1,
\end{split}\end{equation}
and
\begin{align}\label{est22}
 &|\bw_1-\bw_0|\lesssim d^{1\over 4}((\bw_1,s_1), (\bw_0,s_0))\rho^{\frac{1-k}{2}} ((\bw_1,s_1), (\bw_0,s_0)),\nonumber\\
&|s'_1-s'_0|\leq h((\bw_1,s_1), (\bw_0,s_0))\\
&\qquad\qquad+h^{1\over 2} ((\bz,t), (\bw_0,s_0)) h^{1\over 2}((\bw_1,s_1), (\bw_0,s_0)),\nonumber\\ 
&h((\bw_1,s_1), (\bw_0,s_0))\lesssim h ((\bz,t), (\bw_0,s_0)).\nonumber
\end{align}

\color{black}
Since
\begin{align}
    \left\{
                \begin{array}{ll}
                {\displaystyle\partial A\over\displaystyle \partial\bw} + i k\bw^{k-1} \overline{\bz}^k {\displaystyle\partial A \over\displaystyle \partial s} 
=k{\bw}^{k-1}\left(\overline{\bw}^k-\overline{\bz}^k\right)= \mathcal O(\rho^{k-1} h^{1\over2}),\\[12pt]
                   {\displaystyle\partial \P\over\displaystyle \partial\bw} + i k\bw^{k-1} \overline{\bz}^k {\displaystyle\partial \P\over\displaystyle \partial s}
= - {\displaystyle 1\over\displaystyle k}
                   {\displaystyle \P\over\displaystyle A}\left({\displaystyle\partial A\over\displaystyle \partial\bw} + i k\bw^{k-1} \overline{\bz}^k {\displaystyle\partial A\over\displaystyle \partial s}\right)  ,
                \end{array}
              \right.
  \end{align}
 for $0\leq\nu\leq1$, we have
\begin{align*}
     {\displaystyle\partial \widetilde{S}\over\displaystyle \partial\bw'} \bigg|_{(\bz,t;\bw'_\nu, s'_\nu)}
     &=\left({\displaystyle\partial S\over\displaystyle \partial\bw} + i k\bw^{k-1} \overline{\bz}^k {\displaystyle\partial S\over\displaystyle \partial s}\right) \bigg|_{(\bz,t;\bw_\nu, s_\nu)}\\
     &={1\over 4\pi^2}\bigg\{-{k+1\over k} A^{ -{k+1\over k} -1} (1-\P)^{-2}\left({\displaystyle\partial  A\over\displaystyle \partial\bw} + i k\bw^{k-1} \overline{\bz}^k {\displaystyle\partial  A\over\displaystyle \partial s}\right) \bigg|_{(\bz,t;\bw_\nu, s_\nu)}\\
     &\quad
     + 2 A^{-{k+1\over k} } (1-\P)^{-3} \left({\displaystyle\partial  \P\over\displaystyle \partial\bw} + i k\bw^{k-1} \overline{\bz}^k {\displaystyle\partial  \P\over\displaystyle \partial s}\right)\bigg|_{(\bz,t;\bw_\nu, s_\nu)}\bigg\}\\
     &= {1\over 4\pi^2} A^{-{k+1\over k} } (1-\P)^{-2}\bigg[ C_1 {1\over A} 
     +C_2 {\P\over A(1-\P)} \bigg] \left({\displaystyle\partial  A\over\displaystyle \partial\bw} + i k\bw^{k-1} \overline{\bz}^k {\displaystyle\partial  A\over\displaystyle \partial s}\right)\bigg|_{(\bz,t;\bw_\nu, s_\nu)} \\
     %&=S\bigg[ C_1 {1\over A} +C_2 {P\over 1-P} {1\over A} \bigg] \left({\displaystyle\partial  A\over\displaystyle \partial\bw} + i k\bw^{k-1} \overline{\bz}^k {\displaystyle\partial  A\over\displaystyle \partial s}\right)\bigg|_{(\bz,t;\bw_\nu, s_\nu)}\\
     &=S\bigg[ C_1 {1\over A} 
     +C_2 {\P\over A(1-\P)} \bigg]\bigg|_{(\bz,t;\bw_\nu, s_\nu)} k{\bw_\nu}^{k-1}\left(\overline{\bw_\nu}^k-\overline{\bz}^k\right).
\end{align*}

\medskip  
By using \eqref{zw}, \eqref{size}, \eqref {est 1} and \eqref{est11}, we obtain that 
\allowdisplaybreaks{
\begin{align*}
     &\bigg|{\displaystyle\partial \widetilde{S}\over\displaystyle \partial\bw'} \bigg|_{(\bz,t;\bw'_\nu, s'_\nu)} (\bw'_1-\bw'_0)\bigg|\\
     &=\bigg|\left({\displaystyle\partial S\over\displaystyle \partial\bw} + i k\bw^{k-1} \overline{\bz}^k {\displaystyle\partial S \over\displaystyle \partial s}\right) \bigg|_{\color{black}{(\bz,t;\bw_\nu, s_\nu)} }(\bw'_1-\bw'_0)\bigg|\\
      &=\bigg|\left({\displaystyle\partial S\over\displaystyle \partial\bw} + i k\bw^{k-1} \overline{\bz}^k {\displaystyle\partial S \over\displaystyle \partial s}\right) \bigg|_{\color{black}{(\bz,t;\bw_\nu, s_\nu)} }(\bw_1-\bw_0)\bigg|\\
     &\lesssim { 1\over d((\bz,t),(\bw_0, s_0)) } \, \bigg( {1\over |A|} + {|\P|\over h}  \bigg)\bigg|_{(\bz,t;\bw'_\nu, s'_\nu)}
      \left|k{\bw_\nu}^{k-1}\left(\overline{\bw_\nu}^k-\overline{\bz}^k\right)\right|
      {d^{1\over4}( (\bw_1,s_1),(\bw_0,s_0) ) \over  \rho^{k-1\over2}( (\bw_1,s_1),(\bw_0,s_0) )}\\[7pt]
      &\lesssim { 1\over d((\bz,t),(\bw_0, s_0)) } \, \bigg( {1\over |A|} + {1\over h}  \bigg)\bigg|_{(\bz,t;\bw'_\nu, s'_\nu)}
      \left|k{\bw_\nu}^{k-1}\left(\overline{\bw_\nu}^k-\overline{\bz}^k\right)\right|
      {d^{1\over4}( (\bw_1,s_1),(\bw_0,s_0) ) \over  \rho^{k-1\over2}( (\bw_1,s_1),(\bw_0,s_0) )}\\[7pt]
      %&\leq  { C\over d((\bz,t),(\bw_0, s_0)) } \, \bigg( {1\over A} + {1\over h}  \bigg) \bigg|\bigg({\displaystyle\partial  A\over\displaystyle \partial\bw} + i k\bw^{k-1} \overline{\bz}^k {\displaystyle\partial  A\over\displaystyle \partial s} \bigg)\bigg|_{((\bz,t),(\bw_\nu, s_\nu))}\bigg|{d^{1\over4}( (\bw_1,s_1),(\bw_0,s_0) ) \over  \rho^{k-1\over2}( (\bw_1,s_1),(\bw_0,s_0) )}\\[7pt]
  &\lesssim { 1\over d((\bz,t),(\bw_0, s_0)) }\, {1\over  h ((\bz,t),(\bw_0, s_0))} \left|k{\bw_\nu}^{k-1}\left(\overline{\bw_\nu}^k-\overline{\bz}^k\right)\right|      {d^{1\over4}( (\bw_1,s_1),(\bw_0,s_0) ) \over  \rho^{k-1\over2}( (\bw_1,s_1),(\bw_0,s_0) )}\\[7pt]    
     %&\leq { C\over d((\bz,t),(\bw_0, s_0)) }  \bigg|{\displaystyle\partial A\over\displaystyle \partial\bw'}\bigg|\bigg( {1\over A} + {1\over h((\bz,t),(\bw_0, s_0))}  \bigg){d^{1\over4}( (\bw_1,s_1),(\bw_0,s_0) ) \over  \rho^{k-1\over2}( (\bw_1,s_1),(\bw_0,s_0) )}  \\[7pt]
%     &\leq { C\over d((\bz,t),(\bw_0, s_0)) }  \bigg|{\displaystyle\partial A\over\displaystyle \partial\bw'}\bigg|\bigg(  {1\over h((\bz,t),(\bw_0, s_0))}  \bigg)
%      {d^{1\over4}( (\bw_1,s_1),(\bw_0,s_0) ) \over  \rho^{k-1\over2}( (\bw_1,s_1),(\bw_0,s_0) )}  \\[7pt]
%      &= { C\over d((\bz,t),(\bw_0, s_0)) }  \big|   k{\bw_\nu}^{k-1}\left(\overline{\bw_\nu}^k-\overline{\bz}^k\right)\big|\bigg(  {1\over h((\bz,t),(\bw_0, s_0))}  \bigg)
%      {d^{1\over4} ( (\bw_1,s_1),(\bw_0,s_0) )\over  \rho^{k-1\over2}( (\bw_1,s_1),(\bw_0,s_0) )}  \\[7pt]
      &\lesssim { 1\over d((\bz,t),(\bw_0, s_0)) } \cdot   { \rho^{k-1}((\bz,t),(\bw_0, s_0))h^{1\over2}((\bz,t),(\bw_0, s_0))\over h((\bz,t),(\bw_0, s_0))} \cdot
      {d^{1\over4}( (\bw_1,s_1),(\bw_0,s_0) ) \over  \rho^{k-1\over2}( (\bw_1,s_1),(\bw_0,s_0) )}  \\[7pt]
      &= { 1\over d((\bz,t),(\bw_0, s_0)) }  \cdot  { \rho^{k-1\over2}((\bz,t),(\bw_0, s_0))\over d^{1\over4}((\bz,t),(\bw_0, s_0))}  \cdot
      {d^{1\over4}( (\bw_1,s_1),(\bw_0,s_0) ) \over  \rho^{k-1\over2}( (\bw_1,s_1),(\bw_0,s_0) )}  \\[7pt]
      &= { 1\over d((\bz,t),(\bw_0, s_0)) }  \bigg(  { \rho((\bz,t),(\bw_0, s_0))\over  \rho( (\bw_1,s_1),(\bw_0,s_0) ) }  \bigg)^{k-1\over2}
      \bigg({d( (\bw_1,s_1),(\bw_0,s_0) ) \over d((\bz,t),(\bw_0, s_0)) } \bigg)^{1\over4}. \\[7pt]
  %    &\leq { C\over d((\bz,t),(\bw_0, s_0)) }  \bigg(  { d( (\bw_1,s_1),(\bw_0,s_0) ) \over d((\bz,t),(\bw_0, s_0)) } \bigg)^{1\over 2k+2}. \\[7pt]
\end{align*}
}

Moreover, we have

\begin{align*}
     {\displaystyle\partial \widetilde{S}\over\displaystyle \partial\overline{\bw'}} \bigg|_{(\bz,t;\bw'_\nu, s'_\nu)}
     &=\left({\displaystyle\partial S\over\displaystyle \partial\overline{\bw}} - i k\overline{\bw}^{k-1} {\bz}^k {\displaystyle\partial S\over\displaystyle \partial s}\right) \bigg|_{(\bz,t;\bw_\nu, s_\nu)}\\
     &= {1\over 4\pi^2}\bigg\{-{k+1\over k} A^{ -{k+1\over k} -1} (1-\P)^{-2}\left({\displaystyle\partial  A\over\displaystyle \partial\overline{\bw}}  - i k\overline{\bw}^{k-1} {\bz}^k {\displaystyle\partial  A\over\displaystyle \partial s}\right) \bigg|_{(\bz,t;\bw_\nu, s_\nu)}\\
     &\quad
     + 2 A^{-{k+1\over k} } (1-\P)^{-3} \left({\displaystyle\partial  \P\over\displaystyle \partial\overline{\bw}}  - i k\overline{\bw}^{k-1} {\bz}^k {\displaystyle\partial  \P\over\displaystyle \partial s}\right)\bigg|_{(\bz,t;\bw_\nu, s_\nu)}\bigg\}\\
     &=S\bigg\{{ C_1\over A} \left({\displaystyle\partial  A\over\displaystyle \partial\overline{\bw}}  - i k\overline{\bw}^{k-1} {\bz}^k {\displaystyle\partial  A\over\displaystyle \partial s}\right)  \\
     &\quad
    + {C_2\over 1-\P}\left({\displaystyle\partial  \P\over\displaystyle \partial\overline{\bw}}  - i k\overline{\bw}^{k-1} {\bz}^k {\displaystyle\partial  \P\over\displaystyle \partial s}\right)\bigg\} \bigg|_{(\bz,t;\bw_\nu, s_\nu)}.
%     &=A^{-{k+1\over k} } (1-P)^{-2}\bigg[ C_1 {1\over A} 
%     +C_2 {P\over 1-P} {1\over A} \bigg] \left({\displaystyle\partial  A\over\displaystyle \partial\overline{\bw}}  - i k\overline{\bw}^{k-1} {\bz}^k {\displaystyle\partial  A\over\displaystyle \partial s}\right) \bigg|_{(\bz,t;\bw_\nu, s_\nu)} \\
%     %&=S\bigg[ C_1 {1\over A} +C_2 {P\over 1-P} {1\over A} \bigg] \left({\displaystyle\partial  A\over\displaystyle \partial\bw} + i k\bw^{k-1} \overline{\bz}^k {\displaystyle\partial  A\over\displaystyle \partial s}\right)\bigg|_{(\bz,t;\bw_\nu, s_\nu)}\\
%     &=S\bigg[ C_1 {1\over A} 
%     +C_2 {P\over 1-P} {1\over A} \bigg]\bigg|_{(\bz,t;\bw_\nu, s_\nu)} k{\bw_\nu}^{k-1}\left(\overline{\bw_\nu}^k-\overline{\bz}^k\right).
\end{align*}

%\begin{align*}
%     {\displaystyle\partial S\over\displaystyle \partial\overline{\bw'}} \color{black}{ (\bz,t;\bw_{\nu},s_{\nu}) }
%     &= -{ k+1\over k } {\displaystyle\partial A\over\displaystyle \partial\overline{\bw'}} 
%     A^{ -{k+1\over k}-1} (1-P)^{-2} + 2 A^{-{k+1\over k}} (1-P)^{-3}   {\displaystyle\partial P\over\displaystyle \partial\overline{\bw'}} \\[7pt]
%     &=S \color{black}{ (\bz,t;\bw_{\nu},s_{\nu}) }\ \bigg[ C_1 {1\over A}{\displaystyle\partial A\over\displaystyle \partial\overline{\bw'}} + C_2
%     {1\over 1-P} {\displaystyle\partial P\over\displaystyle \partial\overline{\bw'}}  \bigg]\\[7pt]
%     &={1\over d\color{black}{ ((\bz,t), (\bw_{\nu},s_{\nu})) }} \bigg[ C_1 {1\over A}{\displaystyle\partial A\over\displaystyle \partial\overline{\bw'}} + C_2
%     {1\over 1-P} {\displaystyle\partial P\over\displaystyle \partial\overline{\bw'}}  \bigg].\\[7pt]
%\end{align*}
Hence,
\begin{align*}
     &\bigg|{\displaystyle\partial \widetilde{S}\over\displaystyle \partial\overline{\bw'}}\bigg|_{\color{black}{(\bz,t;\bw'_\nu, s'_\nu)}}( \overline{\bw'_1}-\overline{\bw'_0} )\bigg|\\
     &=\bigg|\left({\displaystyle\partial S\over\displaystyle \partial\overline{\bw}} - i k\overline{\bw}^{k-1} {\bz}^k {\displaystyle\partial S\over\displaystyle \partial s}\right) \bigg|_{(\bz,t;\bw_\nu, s_\nu)}(\overline{\bw_1}-\overline{\bw_0})\bigg|\\
     &={ |S|}\bigg|{ C_1\over  A} \left({\displaystyle\partial  A\over\displaystyle \partial\overline{\bw}}  - i k\overline{\bw}^{k-1} {\bz}^k {\displaystyle\partial  A\over\displaystyle \partial s}\right)  
    + {C_2\over 1-\P}\left({\displaystyle\partial  \P\over\displaystyle \partial\overline{\bw}}  - i k\overline{\bw}^{k-1} {\bz}^k {\displaystyle\partial  \P\over\displaystyle \partial s}\right)\bigg| \bigg|_{(\bz,t;\bw_\nu, s_\nu)} 
    \left| \overline{\bw_1}-\overline{\bw_0} \right|\\
     &\lesssim \bigg\{{1\over d}\bigg|{ 1\over  A} \left({\displaystyle\partial  A\over\displaystyle \partial\overline{\bw}}  - i k\overline{\bw}^{k-1} {\bz}^k {\displaystyle\partial  A\over\displaystyle \partial s}\right)  
    \bigg|\bigg\}\bigg|_{(\bz,t;\bw_\nu, s_\nu)}\left| \overline{\bw_1}-\overline{\bw_0} \right|\\
    &\quad+ \bigg\{{1\over d}\bigg| {1\over 1-\P}\left({\displaystyle\partial  \P\over\displaystyle \partial\overline{\bw}}  - i k\overline{\bw}^{k-1} {\bz}^k {\displaystyle\partial  \P\over\displaystyle \partial s}\right)\bigg| \bigg\}\bigg|_{(\bz,t;\bw_\nu, s_\nu)}\left| \overline{\bw_1}-\overline{\bw_0} \right| \\
%          &\lesssim {1\over d\color{black}{ ((\bz,t), (\bw_{\nu},s_{\nu})) }} \bigg| {1\over A}{\displaystyle\partial A\over\displaystyle \partial\overline{\bw'}} ( \overline{\bw_1}-\overline{\bw_0} )\bigg| \\
%          &\quad+ {1\over d\color{black}{ ((\bz,t), (\bw_{\nu},s_{\nu})) }}\bigg|
%     {1\over 1-P} {\displaystyle\partial P\over\displaystyle \partial\overline{\bw'}} ( \overline{\bw_1}-\overline{\bw_0} )  \bigg|\\[7pt]
     &=: I_1+I_2.\\
\end{align*}

%To continue, recall that 
%\begin{align*}
%{\displaystyle\partial \over\displaystyle \partial\overline{\bw'}} ={\displaystyle\partial \over\displaystyle \partial\overline{\bw}} -i k \overline{ \bw}^{k-1}  \bz^k  {\displaystyle\partial \over\displaystyle \partial s}.\\  
%\end{align*}
%
%\noindent Hence
%\begin{align*}
%{1\over A} {\displaystyle\partial A\over\displaystyle \partial\overline{\bw'}}\Bigg| _{ (\bz,t;\bw_{\nu},s_{\nu}) } 
%&={1\over A} \big( k \bw_{\nu}^k \overline{\bw_{\nu}}^{k-1}- ik \overline{\bw_{\nu}}^{k-1} \bz^k\cdot i  \big)\\[7pt]
%&={k\over A}  \overline{\bw_{\nu}}^{k-1}  (\bz^k+\bw_{\nu}^k).\\[7pt]
% \end{align*}
For $I_1$, by \eqref{est22}, \eqref{est11}, \eqref{est 1}, \eqref{est33}, we obtain that 
\begin{align*}
 I_1&\lesssim{1\over d((\bz,t),(\bw_0, s_0)) } \bigg|{k\over A}  \overline{\bw_{\nu}}^{k-1} (\bz^k+\bw_{\nu}^k)\bigg|\ | \overline{\bw_1}-\overline{\bw_0} |  \\[7pt]
&\lesssim {1\over d((\bz,t),(\bw_0, s_0)) } {1\over |A|}  |\bw_{\nu}^{k-1}|  (|\bz|^k+|\bw_{\nu}|^k)\ {d^{1\over4}( (\bw_1,s_1),(\bw_0,s_0) ) \over \rho^{k-1\over2}( (\bw_1,s_1),(\bw_0,s_0) )} \\[7pt]
&\lesssim {1\over d((\bz,t),(\bw_0, s_0)) }\cdot { \rho^{2k-1}((\bz,t),(\bw_0, s_0)) \over \rho^{2k} ((\bz,t),(\bw_0, s_0)) } \cdot {d^{1\over4}( (\bw_1,s_1),(\bw_0,s_0) ) \over \rho^{k-1\over2}( (\bw_1,s_1),(\bw_0,s_0) )} \\[7pt]
&\lesssim {1\over d((\bz,t),(\bw_0, s_0)) } \cdot { \rho^{k-1}((\bz,t),(\bw_0, s_0)) \over h^{1\over 2} ((\bz,t),(\bw_0, s_0)) } \cdot {d^{1\over4}( (\bw_1,s_1),(\bw_0,s_0) ) \over \rho^{k-1\over2}( (\bw_1,s_1),(\bw_0,s_0) )} \\[7pt]
%&={1\over d((\bz,t),(\bw_0, s_0)) } { \rho^{k-1\over2}((\bz,t),(\bw_0, s_0))\over \rho^{k-1\over2} ( (\bw_1,s_1),(\bw_0,s_0) )} \ {d^{1\over4} ( (\bw_1,s_1),(\bw_0,s_0) )\over d^{1\over4}((\bz,t),(\bw_0, s_0))} \\[7pt]
&\approx {1\over d((\bz,t),(\bw_0, s_0)) } \bigg({\rho((\bz,t),(\bw_0, s_0))\over \rho( (\bw_1,s_1),(\bw_0,s_0) )  }\bigg)^{k-1\over2} \, \bigg({ d( (\bw_1,s_1),(\bw_0,s_0) )\over d((\bz,t),(\bw_0, s_0)) }\bigg)^{1\over4}.\\[7pt]
%&\lesssim   {1\over d((\bz,t),(\bw_0, s_0)) }  \bigg( {d( (\bw_1,s_1),(\bw_0,s_0) ) \over d((\bz,t),(\bw_0, s_0))} \bigg)^{1\over 2k+2}.
 \end{align*}

Also, note that 
\begin{align*}
&\left({\displaystyle\partial  \P\over\displaystyle \partial\overline{\bw}}  - i k\overline{\bw}^{k-1} {\bz}^k {\displaystyle\partial  \P\over\displaystyle \partial s}\right)\Bigg|_{ (\bz,t;\bw_{\nu},s_{\nu}) } \\
&= { \bz\over A^{1\over k} } - \bigg( { \bz \cdot \overline{\bw} \over k A^{{1\over k}+1} } \bigg)
\big(  k  \overline{\bw_{\nu}} ^{k-1}\bw_{\nu}^k - ik   \overline{\bw_{\nu}}^{k-1} \bz_{\nu}^k\cdot i \big)\\[7pt]
&=  { \bz\over A^{{1\over k}+1} } \Big\{ |\bz|^{2k} + |\bw_{\nu}|^{2k} - i (t-s_{\nu}) - |\bw_{\nu}|^{2k} - \overline{ \bw_{\nu}}^k  \bz^k \Big\}\\[7pt]
&=  { \bz\over A^{{1\over k}+1} } \Big\{ \bz^{k} ( \overline{\bz}^k - \overline{\bw_{\nu}}^k) - i (t-s_{\nu})  \Big\}\\[7pt]
&=  { \bz\over A^{{1\over k}+1} } \left\{ \re\Big(\bz^{k} ( \overline{\bz}^k - \overline{\bw_{\nu}}^k )\Big) 
- i (t-s_{\nu} +2\im ~(\bz^{k}\overline{\bw_{\nu}}^k) ) -i\im\Big(\bz^{k} ( \overline{\bz}^k - \overline{\bw_{\nu}}^k )\Big)    \right\}.\\[7pt]
\end{align*}

By \color{black}{\eqref{est 1} and \eqref{est11}}, we have
\color{black}
\begin{align*}
 |\bz^{k} ( \overline{\bz}^k - \overline{\bw_{\nu}}^k)| &= 
|\bz^{k}| | \overline{\bz}^k - \overline{\bw_{\nu}}^k | \\
&\lesssim |\bz^{k}| h^{1\over2}((\bz, t),(\bw_0, s_0) )\\
&\lesssim \rho^k((\bz, t),(\bw_0, s_0) ) h^{1\over2}((\bz, t),(\bw_0, s_0) )
\end{align*}
and that 
$$ |(t-s_{\nu} +2\im \bz^{k}\overline{\bw_{\nu}}^k ) |=\sigma ((\bz, t),(\bw_{\nu}, s_{\nu}) )\lesssim h((\bz, t),(\bw_0, s_0) ).  $$
\color{black}
Therefore, we can obtain that 
\begin{align*}
 I_2
% &={1\over d\color{black}{ ((\bz,t), (\bw_{\nu},s_{\nu})) }}\bigg|
%     {1\over 1-P} {\displaystyle\partial P\over\displaystyle \partial\overline{\bw'}} ( \overline{\bw_1}-\overline{\bw_0} )  \bigg|\\[7pt]
  &={1\over d\color{black}{ ((\bz,t), (\bw_{\nu},s_{\nu})) }}\bigg|
     {1\over (1-\P)A} {\displaystyle\bz\over\displaystyle A^{1\over k}} \\[5pt]
     &\quad\times \left\{\re\Big(\bz^{k} ( \overline{\bz}^k - \overline{\bw_{\nu}}^k )\Big) 
- i (t-s_{\nu} +2\im \bz^{k}\overline{\bw_{\nu}}^k ) -i\im\Big(\bz^{k} ( \overline{\bz}^k - \overline{\bw_{\nu}}^k )\Big)    \right\}\bigg| | \overline{\bw_1}-\overline{\bw_0} |  \\[7pt]
&\lesssim {1\over d((\bz,t),(\bw_0, s_0)) } \cdot{ d^{1\over4}( (\bw_1,s_1),(\bw_0,s_0) )\over \rho^{k-1\over2}( (\bw_1,s_1),(\bw_0,s_0) ) }\cdot{\rho((\bz,t),(\bw_0, s_0))\over h((\bz,t),(\bw_0, s_0)) \rho^2((\bz,t),(\bw_0, s_0))} \\
& \quad\times
\left( h^{1\over2}((\bz,t),(\bw_0, s_0))\rho^k((\bz,t),(\bw_0, s_0))+h((\bz,t),(\bw_0, s_0)) \right) \\[7pt]
&\lesssim {1\over d((\bz,t),(\bw_0, s_0)) } \cdot
{ d^{1\over4}( (\bw_1,s_1),(\bw_0,s_0) )\over \rho^{k-1\over2} ( (\bw_1,s_1),(\bw_0,s_0) )}
\cdot
{\rho((\bz,t),(\bw_0, s_0))\over h((\bz,t),(\bw_0, s_0)) \rho^2((\bz,t),(\bw_0, s_0))} \\
& \quad\times h^{1\over2}\left((\bz,t),(\bw_0, s_0)\right)\rho^k((\bz,t),(\bw_0, s_0))\\[7pt]
&= {1\over d((\bz,t),(\bw_0, s_0)) } \cdot{ d^{1\over4}( (\bw_1,s_1),(\bw_0,s_0) )\over \rho^{k-1\over2} ( (\bw_1,s_1),(\bw_0,s_0) )}\cdot {\rho^{k-1}((\bz,t),(\bw_0, s_0))\over h^{1\over2} ((\bz,t),(\bw_0, s_0))} \\[7pt]
&\approx {1\over d((\bz,t),(\bw_0, s_0)) } \cdot \bigg({\rho((\bz,t),(\bw_0, s_0))\over \rho( (\bw_1,s_1),(\bw_0,s_0) )  }\bigg)^{k-1\over2} \, \bigg({ d( (\bw_1,s_1),(\bw_0,s_0) )\over d((\bz,t),(\bw_0, s_0)) }\bigg)^{1\over4}.\\[7pt]
%&\lesssim {1\over d((\bz,t),(\bw_0, s_0)) } \bigg({ d( (\bw_1,s_1),(\bw_0,s_0) )\over d((\bz,t),(\bw_0, s_0)) }\bigg)^{1\over2k+2}.\\[7pt]
 \end{align*}

\color{black}

Moreover, note that
\begin{align*}
 {\displaystyle\partial \widetilde{S}\over\displaystyle\partial s'}\Bigg|_{(\bz,t;\bw'_{\nu},s'_{\nu})}
&= {\displaystyle\partial {S}\over\displaystyle\partial s}\Bigg|_{(\bz,t;\bw_{\nu},s_{\nu})}\\
&=\left\{-{k+1\over k} A^{-{k+1\over k}-1}  {\displaystyle\partial A\over\displaystyle\partial s} (1-\P)^{-2}
+ 2A^{-{k+1\over k}} (1-\P)^{-3}  {\displaystyle\partial \P\over\displaystyle\partial s}\right\}\Bigg|_{(\bz,t;\bw_{\nu},s_{\nu})}\\
&=- {k+1\over k} {i\over A^{{1\over k}+2} (1-\P)^2} - { 2\over k }\ i\ { \bz \overline{\bw_{\nu}}\over A^{ {2\over k}+2 } (1-\P)^3}.
\\[7pt]
 \end{align*}
 \color{black}
As a consequence, by \eqref{est11},  \eqref{est22}, \eqref{dist}, we further obtain that 
\begin{align*}
&\bigg| {\displaystyle\partial \widetilde{S}\over\displaystyle\partial s'} \bigg|_{(\bz,t;\bw'_{\nu},s'_{\nu})}(s'_1-s'_0)\bigg|\\
&={1\over 4\pi^2} \bigg| - {k+1\over k} {i\over A^{{1\over k}+2} (1-\P)^2} - { 2\over k }\ i\ { \bz \overline{\bw_{\nu}}\over A^{ {2\over k}+2 } (1-\P)^3} \bigg|\ |s'_1-s'_0|
\\[7pt]
&= {1\over 4\pi^2A^{{1\over k}+1} (1-\P)^2}   \bigg|   { k+1\over k }  { 1\over A} + { 2\over k  } { \bz \overline{\bw_{\nu}}\over A^{ {1\over k}+1 } (1-\P)}  \bigg|\ |s'_1-s'_0|
\\[7pt]
&\lesssim {1\over d((\bz,t),(\bw_0, s_0))} \bigg(  {1\over |A|} +  { |\bz \overline{\bw_{\nu}}|\over |A|^{ {1\over k} } h((\bz,t),(\bw_0, s_0))}
\bigg)\ |s'_1-s'_0|
\\[7pt]
&\lesssim { 1\over d((\bz,t),(\bw_0, s_0)) } \bigg(  {1\over h((\bz,t),(\bw_0, s_0))} +  {  \rho^2((\bz,t),(\bw_0, s_0))\over \rho^2((\bz,t),(\bw_0, s_0)) h((\bz,t),(\bw_0, s_0))}
\bigg)\ |s'_1-s'_0|
\\[7pt]
&\lesssim { 1\over d((\bz,t),(\bw_0, s_0)) } \ {1\over h((\bz,t),(\bw_0, s_0))} \ |s'_1-s'_0|
\\[7pt]
&\lesssim { 1\over d((\bz,t),(\bw_0, s_0)) } \ {1\over h ((\bz,t),(\bw_0, s_0))} \\
&\quad\times( h ((\bw_1, s_1),  (\bw_0, s_0) ) + h^{1\over2} ((\bz,t),(\bw_0, s_0)) h^{1\over2} ((\bw_1, s_1),  (\bw_0, s_0) )  )
\\[7pt]
&\lesssim { 1\over d((\bz,t),(\bw_0, s_0)) } \ \bigg({h ((\bw_1, s_1),  (\bw_0, s_0) )\over h ((\bz,t),(\bw_0, s_0))} \bigg) ^{1\over2}  
\\[7pt]
&\lesssim { 1\over d((\bz,t),(\bw_0, s_0)) } \ \bigg({d( (\bw_1,s_1),(\bw_0,s_0) )\over d((\bz,t),(\bw_0, s_0))} \bigg) ^{1\over4}   \ \bigg({\rho ((\bz,t),(\bw_0, s_0))\over \rho ((\bw_1, s_1),  (\bw_0, s_0) )} \bigg) ^{k-1\over2} .
\\[7pt]
%&\lesssim { 1\over d((\bz,t),(\bw_0, s_0)) } \ \bigg({d( (\bw_1,s_1),(\bw_0,s_0) )\over d((\bz,t),(\bw_0, s_0))} \bigg) ^{1\over2k+2} . 
 \end{align*}

\color{black}
To sum up, we have
\begin{align*}
&\bigg|\Big \langle \nabla_{ \bw', \overline{\bw'},s'} \widetilde{S}(\bz,t;\bw'_{\nu},s'_{\nu}), \ (\bw'_1-\bw'_0, \overline{\bw'_1}- \overline{\bw'_0}, s'_1-s'_0 ) \Big \rangle \bigg|\\
&\lesssim { 1\over d((\bz,t),(\bw_0, s_0)) }   \ \bigg({\rho ((\bz,t),(\bw_0, s_0))\over \rho ((\bw_1, s_1),  (\bw_0, s_0) )} \bigg) ^{k-1\over2} \ \bigg({d( (\bw_1,s_1),(\bw_0,s_0) )\over d((\bz,t),(\bw_0, s_0))} \bigg) ^{1\over4} .
\end{align*}

\color{black}
From \eqref{dist} we can see that $d((\bz, t),(\bw_0, s_0) )\lesssim\rho^{2k+2}((\bz, t),(\bw_0, s_0) )$. On the one hand, if 
$d((\bz, t),(\bw_0, s_0) )\ll\rho^{2k+2}((\bz, t),(\bw_0, s_0) )$, then $h((\bz, t),(\bw_0, s_0) )\ll\rho^{2k}((\bz, t),(\bw_0, s_0) )$. Therefore, 
$$|z-w_0|\ll\rho((\bz, t),(\bw_0, s_0) ),\quad 
|\sigma((\bz, t),(\bw_0, s_0) )|\ll\rho^{2k}((\bz, t),(\bw_0, s_0) ),$$
which imply that
$$\rho((\bz, t),(\bw_0, s_0) )\approx|w_0|\leq \rho((\bw_1, s_1),(\bw_0, s_0) ).$$
Combing with \eqref{d10}, in this case, we have 
\begin{align*}
    &\bigg|\Big \langle \nabla_{ \bw', \overline{\bw'},s'} \widetilde{S}(\bz,t;\bw'_{\nu},s'_{\nu}), \ (\bw'_1-\bw'_0, \overline{\bw'_1}- \overline{\bw'_0}, s'_1-s'_0 )  \Big \rangle \bigg|\\
     &\lesssim{1\over d((\bz,t),(\bw_0, s_0)) }  \bigg(  { \rho((\bz,t),(\bw_0, s_0))\over  \rho( (\bw_1,s_1),(\bw_0,s_0) ) }  \bigg)^{k-1\over2}
      \bigg({d( (\bw_1,s_1),(\bw_0,s_0) ) \over d((\bz,t),(\bw_0, s_0)) } \bigg)^{1\over4} \\[7pt]
     & \lesssim{1\over d((\bz,t),(\bw_0, s_0)) }       \bigg({d( (\bw_1,s_1),(\bw_0,s_0) ) \over d((\bz,t),(\bw_0, s_0)) } \bigg)^{1\over4} \\[7pt]
     &\lesssim{1\over d((\bz,t),(\bw_0, s_0)) }  \bigg(  { d( (\bw_1,s_1),(\bw_0,s_0) ) \over d((\bz,t),(\bw_0, s_0)) } \bigg)^{1\over 2k+2}. \\[7pt]
\end{align*}

On the other hand, if $d((\bz, t),(\bw_0, s_0) )\approx\rho^{2k+2}((\bz, t),(\bw_0, s_0) )$,  then by \eqref{est33}, we have
\begin{align*}
     &\bigg|\Big \langle \nabla_{ \bw', \overline{\bw'},s'} \widetilde{S}(\bz,t;\bw'_{\nu},s'_{\nu}), \ (\bw'_1-\bw'_0, \overline{\bw'_1}- \overline{\bw'_0}, s'_1-s'_0 )  \Big \rangle \bigg|\\
     &\lesssim{1\over d((\bz,t),(\bw_0, s_0)) }  \bigg(  { \rho((\bz,t),(\bw_0, s_0))\over  \rho( (\bw_1,s_1),(\bw_0,s_0) ) }  \bigg)^{k-1\over2}
      \bigg({d( (\bw_1,s_1),(\bw_0,s_0) ) \over d((\bz,t),(\bw_0, s_0)) } \bigg)^{1\over4} \\[7pt]
      &={ 1\over d((\bz,t),(\bw_0, s_0)) }   \bigg(  {  d^{1\over 2k+2}( (\bw_1,s_1),(\bw_0,s_0) )\over  \rho( (\bw_1,s_1),(\bw_0,s_0) ) }  \bigg)^{k-1\over2}
      \bigg(  { \rho((\bz,t),(\bw_0, s_0))\over d^{1\over 2k+2}( (\bz,t),(\bw_0,s_0) ) }  \bigg)^{k-1\over2}\times\\
      &\quad\times\bigg(  { d( (\bw_1,s_1),(\bw_0,s_0) ) \over d((\bz,t),(\bw_0, s_0)) } \bigg)^{1\over 2k+2}\\
         &\leq { 1\over d((\bz,t),(\bw_0, s_0)) }  \bigg(  { d( (\bw_1,s_1),(\bw_0,s_0) ) \over d((\bz,t),(\bw_0, s_0)) } \bigg)^{1\over 2k+2}. \\[7pt]
\end{align*}

\color{black}
Consequently, for 
$$ d((\bw_1, s_1),  (\bw_0, s_0) )\leq  c d((\bz,t),(\bw_0, s_0)),$$
with some small $c$, we have
\begin{align*} 
&|S(\bz,t;\bw_1, s_1)-S(\bz,t;\bw_0, s_0)|\\[7pt]
&=\left|\widetilde{S}(\bz,t;\bw'_1, s'_1) -\widetilde{S}(\bz,t;\bw'_0, s'_0)\right|\\
&= \Bigg| \int_0^1\Big \langle \nabla_{ \bw', \overline{\bw'},s'} \widetilde{S}(\bz,t;\bw'_{\nu},s'_{\nu}), \ (\bw'_1-\bw'_0, \overline{\bw'_1}- \overline{\bw'_0}, s'_1-s'_0 )  \Big \rangle d\nu \Bigg|\\[7pt] 
&\lesssim \int_0^1 { 1\over d((\bz,t),(\bw_0, s_0)) } \ \bigg({d( (\bw_1,s_1),(\bw_0,s_0) )\over d((\bz,t),(\bw_0, s_0))} \bigg) ^{1\over2k+2}   d\nu\\[7pt]
 &\lesssim  { 1\over d((\bz,t),(\bw_0, s_0)) } \ \bigg({d( (\bw_1,s_1),(\bw_0,s_0) )\over d((\bz,t),(\bw_0, s_0))} \bigg) ^{1\over2k+2}.   \\[7pt]
 \end{align*}

This finishes the proof of Theorem \ref{thm 7.1}.
\end{proof}
%\color{black}

\end{document}